\documentclass{amsart}

\usepackage{amssymb,amsmath,amsfonts,amsthm,epsfig,amscd}
\usepackage{commath}
\usepackage{stmaryrd}
\usepackage[all,cmtip,poly]{xy}
\usepackage{svn}
\usepackage{mathtools}
\usepackage{mathrsfs}
\usepackage{tikz-cd}
\usepackage{hyperref}
\usepackage{nameref}
\usepackage{dsfont}
\usepackage{stackrel}

\newtheorem{lem}{Lemma}[section]

\newtheorem{df}[lem]{Definition}

\newtheorem{cor}[lem]{Corollary}
\newtheorem{thm}[lem]{Theorem}

\newtheorem{prop}[lem]{Proposition}

\newtheorem{conj}[lem]{Conjecture}

\theoremstyle{remark}
\newtheorem{rem}[lem]{Remark}

\newcommand{\SL}{\mathrm{SL}}
\newcommand{\Sym}{\mathrm{Sym}}

\newcommand{\C}{\CC}

\renewcommand{\AA}{{\mathbb A}}

\def\rank{\mathrm{rank}}
\def\Ann{\mathrm{Ann}}
\def\wAnn{\widetilde{\Ann}}
\def\Ug{U(\frak{g})}
\def\wU{\widehat{U}}
\def\wUg{\widehat{\Ug}}
\def\wUk{\widehat{U}_\bk}
\def\wUgK{\widehat{U(\frak{g}_K)}}

\def\wtM{\widetilde{M}}
\def\whM{\widehat{M}}
\def\wH{\widetilde{H}}

\def\Ind{\mathrm{Ind}}

\def\gr{\mathrm{gr}}

\def\Hom{\mathrm{Hom}}

\def\for{\mathrm{for}}

\def\End{\mathrm{End}}

\newcommand{\la}{\mathrm{la}}
\newcommand{\an}{\mathrm{an}}
\newcommand{\Zg}{Z(\frak{g})}
\newcommand{\Zgo}{Z(\frak{g}_0)}

\newcommand{\BB}{{\mathbb B}}
\newcommand{\CC}{{\mathbb C}}

\newcommand{\FF}{{\mathbb F}}
\newcommand{\GG}{{\mathbb G}}

\newcommand{\NN}{{\mathbb N}}

\newcommand{\QQ}{{\mathbb Q}}
\newcommand{\RR}{{\mathbb R}}

\newcommand{\TT}{{\mathbb T}}

\newcommand{\ZZ}{{\mathbb Z}}

\newcommand{\bk}{\ensuremath{\mathbf{k}}}

\newcommand{\cL}{{\mathcal L}}

\newcommand{\cO}{{\mathcal O}}

\makeatletter
\tikzset{
  column sep/.code=\def\pgfmatrixcolumnsep{\pgf@matrix@xscale*(#1)},
  row sep/.code   =\def\pgfmatrixrowsep{\pgf@matrix@yscale*(#1)},
  matrix xscale/.code=%
    \pgfmathsetmacro\pgf@matrix@xscale{\pgf@matrix@xscale*(#1)},
  matrix yscale/.code=%
    \pgfmathsetmacro\pgf@matrix@yscale{\pgf@matrix@yscale*(#1)},
  matrix scale/.style={/tikz/matrix xscale={#1},/tikz/matrix yscale={#1}}}
\def\pgf@matrix@xscale{1}
\def\pgf@matrix@yscale{1}
\makeatother

\title{Sharp bounds for multiplicities of Bianchi modular forms}
\begin{document}
\author{Weibo Fu}
\address{Department of Mathematics, Princeton University, Princeton, NJ, USA.}
\email{wfu@math.princeton.edu}

\begin{abstract}
We prove a degree-one saving bound for the dimension of the space of cohomological automorphic forms of fixed level and growing weight on $\SL_2$ over any number field that is not totally real.
In particular, we establish a sharp bound on the growth of cuspidal Bianchi modular forms.
We transfer our problem into a question over the completed universal enveloping algebras by applying an algebraic microlocalisation of Ardakov and Wadsley to the completed homology.
We prove finitely generated Iwasawa modules under the microlocalisation are generic, solving the representation theoretic question by estimating growth of Poincar\'e–Birkhoff–Witt filtrations on such modules.
\end{abstract}
\maketitle
\tableofcontents

\section{Introduction}\label{intro}
Let $F$ be a number field of degree $r = r_1 + 2r_2$, with $r_1$ real places and $r_2$ complex places. 
Let $F_{\infty}=F\otimes_{\QQ}\RR$, so that $\SL_2(F_{\infty})=\SL_2(\RR)^{r_1}\times\SL_2(\CC)^{r_2}$.
Let $Z_{\infty}$ be the centre of $\SL_2(F_{\infty})$, $K_f$ be a compact open subgroup of $\SL_2(\mathbb{A}^{\infty}_F)$ and let 
\[X(K_f) :=\SL_2(F)\backslash \SL_2(\mathbb{A}_F)/K_fZ_{\infty}.\] 
If ${\bk}=(k_1,...,k_{r_1+r_2})$ is an $(r_1+r_2)$-tuple of positive even integers, we define $W_{\bk}$ to be the representation of $\SL_2(F_{\infty})$ obtained by taking the tensor product of the representation $\Sym^{k_i-2}$ of $\SL_2(F_{v_i})$ when $v_i$ is a real place and the representation $\Sym^{k_i/2-1} \otimes \overline{\Sym}^{k_i/2-1}$ of $\SL_2(F_{v_i})$ when $v_i$ is a complex place.
We also use $W_{\bk}$ to denote the local system on $X({K_f})$ coming from the representation $W_{\bk}$.
We let $S_{\bk}(K_f)$ be the space of cohomological cusp forms on $X(K_f)$ with weight ${\bk}$. 
We define $\Delta(\bk)$ to be
\[\Delta({\bk})=\prod_{1\leq i\leq r_1}k_i\times \prod_{r_1<i\leq r_1+r_2} k_i^2.\]

In this paper, we will adapt $p$-adic algebraic methods to study the growth of dimension of $S_{\bk}(K_f)$ as $\bk$ varies and $K_f$ is fixed.

When $F$ is totally real, Shimizu \cite{Shi63} has proven that \[\dim_{\C} S_{\bk}(K_f)\sim C\cdot \Delta({\bk})\]
for some constant $C$ independent of ${\bk}$.

When $F$ is not totally real, the growth rate of $\dim_{\C} S_{\bk}(K_f)$ is wildly open.
The first nontrivial bound is given by a trace formula method:
\begin{eqnarray}\label{TF bound}
\dim_{\C} S_{\bk}(K_f) = o(\Delta(\bk)). \end{eqnarray}
\begin{conj}\label{imag quad conj}
If $F$ is imaginary quadratic, $\bk = (k)$, there exists a constant $c$ depending only on $K_f$ such that for $k \geq 1$,  \[ \dim_{\C} S_{\bk}(K_f) \leq c \cdot k. \]
\end{conj}
This conjecture is supported by experimental data of Finis-Grunewald-Tirao \cite{FGT10} and the work of Calegari-Mazur \cite{CM09} (for Hida families).
Under mild conditions, such an upper bound of linear growth rate is sharp from the base change of classical elliptic modular forms.

In this paper, we prove this conjecture by giving a polynomial saving improvement of (\ref{TF bound}).

If $F$ is imaginary quadratic, Finis, Grunewald and Tirao \cite{FGT10} established the bounds
\[k \ll \dim_{\C} S_{\bk}(K_f)\ll_{K_f} \frac{k^2}{\ln k},\ \ {\bk}=(k)\]
for suitable $K_f$ using base change and the trace formula respectively. 
In \cite{Mar12}, Marshall improved (\ref{TF bound}) by a power saving bound: suppose $\bk = (k,\cdots,k)$ is parallel, \[\dim_{\C} S_{\bk}(K_f)\ll_{\epsilon,K_f}k^{r-1/3+\epsilon}.\] 
And later on in \cite{Hu21}, Hu proved a better power saving bound \[\dim_{\C} S_{\bk}(K_f)\ll_{\epsilon,K_f}k^{r-1/2+\epsilon}.\]

Both of them use mod $p$ representation theory methods applied to Emerton's completed homology and relate the completed homology back to $S_{\bk}(K_f)$ via a spectral sequence and the Eichler-Shimura isomorphism summarized as \begin{eqnarray}\label{ES} \dim_{\CC} H_{c}^{r_{1}+r_{2}} (Y(K_f), W_{\bk})=2^{r_{1}} \dim_{\CC} S_{\bk}(K_{f}) \end{eqnarray} in \cite{Mar12}.
Here $Y(K_f)$ is defined to be \[ Y(K_f) := X(K_f) / K_{\infty}, \]
and $K_{\infty}$ is the maximal compact subgroup of $\SL_2(F_{\infty})$.

If $F$ is imaginary quadratic, $Y(K_f)$ is a hyperbolic 3-manifold. The space $S_{\bk}(K_f)$ consists of cuspidal Bianchi modular forms of parallel weight $(k,k)$.
This space corresponds to the first compactly supported cohomology with coefficient local system $W_{\bk} = \Sym^{k/2-1} \otimes \overline{\Sym}^{k/2-1}$ by the Eichler-Shimura isomorphism (\ref{ES}).

The main result of our paper is to bound $\dim_{\C} S_{\bk}(K_f)$: 

\begin{thm}\label{main uncond}
If $F$ is not totally real, then for any fixed $K_f$, we
have
\begin{eqnarray}\label{upp bound}
\dim_{\C} S_{\bk}(K_f) \leq_{K_f} (\min_{1 \leq i \leq r} k_i)^{-1} O(\Delta(\bk)).
\end{eqnarray}
If moreover $\bk = (k,\cdots,k)$ is parallel, we have
\[\dim_{\C} S_{\bk}(K_f) \leq_{K_f} O(k^{r-1}).\]
\end{thm}

\begin{cor}
The conjecture \ref{imag quad conj} is correct. Suppose $K_f$ is sufficiently small, for the arithmetic hyperbolic 3-manifold $Y(K_f)$ and cohomological degree $n=1,2$, we have the sharp bounds
\[ \dim_{\CC} H_{c}^n (Y(K_f), W_{\bk}) \sim_{K_f} k. \]
\end{cor}

Compared to \cite[Thm 1, Cor 2]{Mar12} and \cite[Thm 1.1]{Hu21}, we get a degree-one saving bound and we do not need $\epsilon$ weakening.
Let $F_p = F \otimes_{\QQ} \QQ_p$ with ring of integers $\cO_{p}$.
It is very worth noting that both of \cite[Thm 1, Cor 2]{Mar12} and \cite[Thm 1.1]{Hu21} crucially use the $\SL_2(F_p)$-action on the completed homology, but we only make uses of the group action of the first congruence subgroup of $\SL_2(\cO_p)$.

If $F$ only admits one complex place, or equivalently $r_2 = 1$, it seems likely (\ref{upp bound}) gives a sharp upper bound by heuristics from the Calegari-Emerton conjecture.

Let $K$ be a finite extension of $\QQ_p$. 
Now we fix a compact open level subgroup $G \subset \SL_2(\cO_{p})$.
If $K_f$ further decomposes as \[ K_f = K_p K^p ~ \for ~ K_p \subset G, ~ K^p \subset \SL_2(\AA^{p,\infty}), \]
we introduce completed homology of tame level $K^p$ (see \cite{CE09}) as
\begin{eqnarray}\label{comp homo} \wH_{\bullet}(K^p):=\varprojlim_{s} \varprojlim_{K_p \subset G} H_{\bullet} (Y(K_pK^p), \ZZ / p^{s} \ZZ) \otimes_{\ZZ_p} K. \end{eqnarray}
They are finitely generated modules over the Iwasawa algebra $K[[G]]$ (see \cite{Eme06}).

To prove Thm \ref{main uncond}, we establish the following theorem on sub-polynomial growth of algebraic quotients and their higher cohomology of a finitely generated Iwasawa $\QQ_p[[G]]$-module for $G$ being a product of uniform (\cite[\S 4]{DDSMS99}) pro-$p$ compact open subgroups of $\SL_2(\ZZ_p)$.
The main theorem therefore can be obtained by using a fundamental spectral sequence due to Emerton \cite{Eme06}, \cite{CE09} (also see \cite{Mar12} and \cite{Hu21}).

By \cite{Ven02}, the Iwasawa algebra $K[[G]]$ is Auslander regular.
In particular, it is of finite global dimension. 
The $G$-homology of any Iwasawa module should be vanishing above a degree which only depends on $G$.
A multiplicity-free polynomial $p_{\wtM}$ of $\bk$ is a polynomial such that for each monomial term $k_1^{g_1}\cdots k_r^{g_r}$ of $p_{\wtM}$, each $g_i$ is at most 1 for $1 \leq i \leq r$.

For each $1 \leq i \leq r$, let \[ G_i := (I_2 + p \mathrm{M}_2(\ZZ_p)) \cap \SL_2(\ZZ_p). \]
Let $G = \prod\limits_{i=1}^r G_i$, $\bk = (k_1,\cdots,k_r) \in \NN^r$, and $W_{\bk}$ be the algebraic representation $\boxtimes_{i=1}^{r} \Sym^{k_i}$ of $G$.

\begin{thm}\label{main loc thm}
Let $\wtM$ be a finitely generated $\QQ_p[[G]]$-module of rank $d$.
There exists a multiplicity-free polynomial $p_{\wtM}$ of $\bk$ of degree at most $r-1$ associated to $\wtM$ such that
for any $\bk \in \NN^r$ and any $i \geq 1$, 
\[ | \dim_{\QQ_p} H_0(G, \wtM \otimes W_\bk) - d \prod_{i=1}^r (k_i+1) | \leq p_{\wtM}(\bk), \]
\[ \dim_{\QQ_p} H_i(G, \wtM \otimes W_\bk) \leq p_{\wtM}(\bk). \]
\end{thm}

Let us first consider the growth of algebraic representations for some simplest $K[[G]]$-modules.

If $\wtM$ is of canonical dimension $0$, then $\wtM$ is finite dimensional by \cite[Lem 10.13]{AW13}. The polynomials can be chosen to be constants (of degree $0$).

If $\wtM \simeq K[[G]]$ as the module over itself, in section \S \ref{loc app}, we explicitly exhibit \[ \Hom_{K[[G]]} (K[[G]], W_{\bk}) \simeq \Hom_{K[[G]]} (\End_K(W_{\bk}), W_{\bk}). \]
Therefore \[ \dim_K \Hom_{K[[G]]} (K[[G]], W_{\bk}) = \dim_K W_{\bk} = \prod_{i=1}^r (k_i+1). \]
This is an analogue of the classical algebraic Peter-Weyl theorem for $\SL_2$.

Let $\frak{g}_K$ be the Lie algebra of $G$ with coefficients in $K$.
To deduce Thm \ref{main loc thm}, we pass our problem to an algebraic microlocalisation of the Iwasawa algebra via a completed universal enveloping algebra $K[[G]] \to \widehat{U(\frak{g}_K)}$ introduced in \cite{AW13}. 
In \S \ref{loc app}, we further apply homological degree-shifting arguments to reduce to only treat the degree-zero case for a cyclic torsion Iwasawa module $\wtM$.

The completed enveloping algebra $\widehat{U(\frak{g}_K)}$ is the $p$-adic completion of the usual universal enveloping algebra $U(\frak{g}_K)$. 
$W_\bk$ admits actions of $\wUgK$ and $U(\frak{g}_K)$.

Although the Iwasawa algebra has a very small center \cite{Ard04}, $\widehat{U(\frak{g}_K)}$ has a larger center containing Casimir operators $\Delta_1,\cdots,\Delta_r$.
If $K = \QQ_p$ and $\lambda = (\lambda_1,\cdots,\lambda_r) \in \ZZ_p^r$, we may specialize $\Delta_i$ to be $\lambda_i$ for $1 \leq i \leq r$.
We use $\wUg_\lambda$ to denote this specialization, the goal of \S \ref{specialization} is to prove the following theorem.
\begin{thm}\label{main spec}
For any $\lambda \in \ZZ_p^r$, the following map is injective \begin{eqnarray}\label{Iwa to spec} \QQ_p[[G]] \hookrightarrow \wUg_\lambda. \end{eqnarray}
\end{thm}

The right hand side of (\ref{Iwa to spec}) as a Noetherian algebra has a smaller dimension compared to the left hand side, a priori the kernel of it is a two-sided ideal. But it surprisingly turns out to be zero.

After we completed this paper, Konstantin Ardakov pointed it out to us that Theorem \ref{main spec} can be deduced from the main results of \cite[Thm 4.6, Thm 5.4]{AW14}.
Since our proof is different and the intermediate results may be of independent interest, we still include \S \ref{specialization} as part of the paper.

A \emph{generic} element $\delta \in \wUg$ is an element such that the image of $\delta$ under the specialization $\wUg \twoheadrightarrow \wUg_\lambda$ is non-zero for any $\lambda \in \ZZ_p^r$.
Theorem \ref{main spec} asserts that the image of the Iwasawa algebra via the microlocalisation consists of generic elements (\S \ref{gro enve}) of the completed enveloping algebra.

\begin{thm}\label{growth}
Let $\whM$ be a cyclic torsion module over $\wUgK$ with a generator killed by a generic element. 
There exists a multiplicity-free polynomial $p_{\whM}$ in $r$ variables of degree at most $r-1$ such that 
for any $\bk \in \NN^r$
\[ \dim_K \Hom_{\wUgK}(\whM, W_\bk) \leq p_{\whM}(\bk) ~ \for ~ \bk \in \NN^r. \]
\end{thm}

We will prove some comparison results identifying $W_{\bk}$-quotients of the original Iwasawa module $\wtM$ and $W_{\bk}$-quotients of its microlocalisation $\whM = \wUgK \otimes_{K[[G]]} \wtM$. 
Thm \ref{main loc thm} will be deduced from combining Thm \ref{growth} and Thm \ref{main spec}. 

Finally, we prove Thm \ref{growth} by estimating the growth of dimension of a Poincar\'e–Birkhoff–Witt filtration on $\whM$. 
There is a natural integral model $\widehat{U(\frak{g}_0)}$ of $\wUgK$. It is important for us to consider the image of $\widehat{U(\frak{g}_0)}$ in $\End(W_\bk)$.
And the PBW filtration is linked to multiplicities by Prop \ref{sl2 end}.

\subsection*{Acknowledgments}
I would like to thank Yongquan Hu for helpful discussions. 
I want to thank Richard Taylor for checking the details of this paper and catching an error in our earlier draft.
I would also like to thank Konstantin Ardakov, Shilin Lai, Lue Pan, Zicheng Qian, and Peter Sarnak for interesting conversations about this paper, as well as the anonymous referee for expository corrections.

\section{Results on algebra}\label{Algebra}

The rings of interest will always be Noetherian.

\begin{df}\label{rank}
If $A$ is an integral domain (not necessarily commutative), and $M$ is a finitely generated $A$-module. 
The field of fractions $\cL$ of $A$ is a division ring that is flat over $A$ on both sides.
We use $\dim_{\cL} \cL \otimes_{A} M$ to denote the \emph{rank} of $M$.
\end{df}

Let $\Lambda$ be a partially ordered abelian group.
A $\Lambda$-filtration $F_\bullet A$
on a ring $A$ is a set $\{F_{\lambda} A | \lambda \in \Lambda \}$ of additive subgroups of $A$ such that
\begin{itemize}
\item $1 \in F_0 A$,
\item $F_{\lambda} A \subset F_{\mu} A$ whenever $\lambda \leq \mu$,
\item $F_{\lambda} A \cdot F_{\mu} A \subset F_{\lambda+\mu} A$ for all $\lambda, \mu \in \Lambda$.
\end{itemize}
The filtration on $A$ is said to be \emph{separated} if $\bigcap_{\lambda \in \Lambda} F_{\lambda} A=\{0\}$, and it is said to be \emph{exhaustive} if $\bigcup_{\lambda \in \Lambda} F_{\lambda} A=A$. 

In a similar way, given a $\Lambda$-filtered ring $F_\bullet A$ and an $A$-module $M$, a filtration of $M$ is a set $\{F_\lambda M | \lambda \in \Lambda \}$ of additive subgroups of $M$ such that
\begin{itemize}
\item $F_{\lambda} M \subset F_{\mu} M$ whenever $\lambda \leq \mu$,
\item $F_{\lambda} A \cdot F_{\mu} M \subset F_{\lambda+\mu} M$ for all $\lambda, \mu \in \Lambda$.
\end{itemize}
Again, the filtration on $M$ is said to be \emph{separated} if $\bigcap_{\lambda \in \Lambda} F_{\lambda} M = \{0\}$, and it is said to be \emph{exhaustive} if $\bigcup_{\lambda \in \Lambda} F_{\lambda} M = M$.

If $\Lambda \subset \RR$, we can define graded rings and modules for the $\Lambda$-filtration.
Let $A$ be a $\Lambda$-filtered ring. 
For any $\lambda \in \Lambda$, we put
\[ F_{\lambda-} A := \bigcup_{s < \lambda} F_s A. \]
The \emph{associated graded ring} is defined to be
\[ \gr(A) := \bigoplus_{\lambda \in \Lambda} F_\lambda A / F_{\lambda-} A.\]
Given a filtered $F_\bullet A$ module $F_\bullet M$, we similarly define \emph{associated graded module} ($F_{\lambda-} M$ is similarly defined)
\[ \gr(M) := \bigoplus_{\lambda \in \Lambda} F_\lambda M / F_{\lambda-} M.\]
$\gr(M)$ is a natural $\gr(A)$-module.
For any $m \in F_\lambda M \backslash F_{\lambda-} M$, we use \[ \gr(m) \in \frac{F_\lambda M}{F_{\lambda-} M} \subset \gr(M) \] to denote the corresponding principal symbol. 

\begin{lem}\label{grade ring quotient generator}
Let $A$ be a $\Lambda$-filtered ring with $\Lambda \subset \RR$. 
Suppose for any non-zero $x \in A$, there exists $\lambda_x \in \Lambda$ such that $x \in F_{\lambda_x} A \backslash F_{\lambda_x -} A$.
If $\frak{a} \in F_a A \backslash F_{a-} A$ for $a \in \Lambda$, and $\gr(\frak{a})$ is a non-zero divisor of $\gr(A)$, we have
\[ \gr(A / A\cdot \frak{a}) \simeq \gr(A) / \gr(A) \cdot \gr(\frak{a}), \]
where the filtration on $A / A\cdot \frak{a}$ is induced from the filtration on $A$.
\end{lem}
\begin{proof}
By the assumptions, $\frak{a} \in A$ is a non-zero divisor.
And for any $\lambda \in \Lambda$, we have
\[ A \cdot \frak{a} \cap F_\lambda A = F_{\lambda-a} A \cdot \frak{a}, ~ A \cdot \frak{a} \cap F_{\lambda-} A = F_{(\lambda-a)-} A \cdot \frak{a}. \]
By definition of the induced filtration,
\[ F_\lambda (A/A\cdot \frak{a}) = F_\lambda A / F_{\lambda-a} A \cdot \frak{a}. \]
For saving notation, let \[ \gr_\lambda A := F_\lambda A / F_{\lambda-} A ~ \text{and} ~ \gr_\lambda(A/A \cdot \frak{a}) := F_\lambda (A/A \cdot \frak{a}) / F_{\lambda-} (A/A \cdot \frak{a}). \]
We have the following commutative diagram with exact rows and exact columns
\[ \xymatrix{
& 0  & 0 & 0 \\
0 \ar[r] & F_{\lambda-} (A / A \cdot \frak{a}) \ar[r] \ar[u] & F_{\lambda} (A / A \cdot \frak{a}) \ar[r] \ar[u] & \gr_\lambda (A / A \cdot \frak{a}) \ar[r] \ar[u] & 0 \\
0 \ar[r] &F_{\lambda-} A \ar[r] \ar[u] & F_{\lambda} A \ar[r] \ar[u] & \gr_\lambda A \ar[r] \ar[u] & 0 \\
0 \ar[r] & F_{(\lambda-a)-} A \cdot \frak{a} \ar[r] \ar[u] & F_{(\lambda-a)} A \cdot \frak{a} \ar[r] \ar[u] & \gr_{\lambda-a} A \cdot \gr(a) \ar[r] \ar[u] & 0 \\
& 0 \ar[u] & 0 \ar[u] & 0. \ar[u] \\
} \]
The claim follows from the last vertical short exact sequence applying the snake lemma.
\end{proof}

\begin{lem}\label{ext comp}
Let $A$ be a Noetherian ring with a two-sided ideal $I$ such that $A / I$ is $p$-torsion free.
The $p$-adic completion $\hat{A}$ of $A$ has a two-sided ideal \[ \hat{I} := \varprojlim_{a} \left(\frac{(p^a)+I}{(p^a)}\right) = \hat{A} \cdot I = I \cdot \hat{A} \subset \hat{A}.\]
And $\hat{A} / \hat{I}$ is isomorphic to the $p$-adic completion of $A / I$.
\end{lem}
\begin{proof}
We have to prove $\varprojlim_{a} \left(\frac{(p^a)+I}{(p^a)}\right) \subset \hat{A} \cdot I$, the other inclusion is clear.

Suppose $I$ is generated by $\{m_1,\cdots,m_l\} \subset I$.
Choose any compatible system \[ (\overline{m^{(a)}}) \in \varprojlim_{a} \left(\frac{(p^a)+I}{(p^a)}\right), ~ m^{(a)} \in I. \] 
We use induction to choose coefficients of $m_i$ in the expression of $\overline{m^{(a)}}$.

For $a = 1$, we pick coefficients $\{x^{(1)}_i \in A\}$ such that $\overline{m^{(1)}} =  \sum\limits_{i=1}^l \overline{x^{(1)}_i} \bar{m}_i$.
For a given $a \in \NN$, suppose we have constructed $\{x^{(a)}_i \in A\}$ and express \[ \overline{m^{(a)}} = \sum\limits_{i=1}^l \overline{x^{(a)}_i} \bar{m}_i \in \frac{(p^a)+I}{(p^a)}, \] we want to inductively construct $\{x^{(a+1)}_i \in A\}$.
We lift $\overline{m^{(a+1)}_i}$ to $\tilde{m} \in I$. 
Since $A / I$ is $p$-torsion free, there exists $\{d_i^{(a+1)} | 1 \leq i \leq l\}$ such that \[ \tilde{m}-(\sum\limits_{i=1}^l x^{(a)}_i m_i) = p^a \sum\limits_{i=1}^l d^{(a+1)}_i m_i.\] 
Let $\overline{x^{(a+1)}_i} := \overline{x^{(a)}_i} + p^a \overline{d^{(a+1)}_i} ~ \mathrm{mod} ~ p^{a+1}$ for $1 \leq i \leq l$.
The lift $\overline{m^{(a+1)}_i}$ can be expressed as \[ \overline{m^{(a+1)}_i} = \sum\limits_{i=1}^l \overline{x^{(a+1)}_i} \bar{m}_i \in \frac{(p^{a+1})+I}{(p^{a+1})}. \]
Therefore for each $1 \leq i \leq l$, $(\overline{x^{(a)}_i})$ defines an element in $\hat{A}$, $\hat{I} \subset \hat{A} \cdot I$.
Similarly $\hat{I} = I \cdot \hat{A}$.

Consider the short exact sequence of inverse systems:
\[ 0 \to \left(\frac{(p^a)+I}{(p^a)}\right) \to \left(\frac{A}{(p^a)}\right) \to \left(\frac{A}{(p^a)+I}\right) \to 0, \]
the system $\left(\frac{(p^a)+I}{(p^a)}\right)$ satisfies the Mittag-Leffler condition,
the inverse limits give a short exact sequence by \cite[02MY Lem 12.31.3]{Sta}.
$\hat{A} / \hat{I}$ is isomorphic to the $p$-adic completion of $A / I$.
\end{proof}

\section{Growth of algebraic quotients for (completed) enveloping algebras}\label{gro enve} 

Let $K$ be a field with a subring $R$ such that $2$ is invertible in $R$.
Let $\frak{g}_0$ be a direct sum of $R$-Lie algebras with $K$-valued extension $\frak{g}$ $$\frak{g}_0 := \bigoplus_{i=1}^r \frak{s}\frak{l}_{2,R}, ~~ \frak{g} := \bigoplus_{i=1}^r \frak{s}\frak{l}_{2,K} = K \otimes_R \frak{g}_0.$$
The universal enveloping algebra $U(\frak{g}_0) = \bigotimes_{i=1}^r U(\frak{s}\frak{l}_{2,R})$ (and similarly for $U(\frak{g})_{/ K}$) is a multi-filtered $R$-algebra with index group $\Lambda = \ZZ^r$.
To be precise, let $h_i, e_i, f_i$ be the basis of $\frak{s}\frak{l}_{2,R}$ with the relations \begin{equation}\label{sl2 rel} 
 [h_i, e_i]=2 e_i, \quad[h_i, f_i]=-2 f_i, \quad[e_i, f_i]=h_i . 
\end{equation}

Letting \[ \Delta_i := \frac{1}{2} h_i^2+ e_i f_i +  f_i e_i = \frac{1}{2} h_i^2-h_i+2e_i f_i\] be the Casimir operator for $i$-th component, we are interested in the polynomial ring $K[\Delta_1,\cdots,\Delta_r] \subset Z(U(\frak{g}))$. 

For $\lambda = (\lambda_1,\cdots,\lambda_r) \in R^r$, we define
 \[\Ann_Z(\lambda)^\circ := \sum\limits_{i=1}^r Z(U(\frak{g}_0)) \cdot (\Delta_i - \lambda_i), ~ \Ann_Z(\lambda) := K \cdot \Ann_Z(\lambda)^\circ \]
as ideals of $Z(U(\frak{g}_0))$ and $Z(U(\frak{g}))$, giving rise to the extension ideal and quotient ring of $U(\frak{g}_0)$ and $\Ug$ \[ U_Z(\lambda)^\circ := U(\frak{g}_0) \cdot \Ann_Z(\lambda)^\circ, ~ U_Z(\lambda) := \Ug \cdot \Ann_Z(\lambda), \]
\begin{eqnarray}\label{U lambda} U_\lambda^\circ := U(\frak{g}_0) / U_Z(\lambda)^\circ, ~ U_\lambda := \Ug / U_Z(\lambda). \end{eqnarray}

Let $p$ be an odd prime. If $K$ is a finite extension of $\QQ_p$ with ring of integers $R$, we define 
\[ \widehat{U(\frak{g}_0)} := \varprojlim_{a} \left( \frac{U(\frak{g}_0)}{p^a U(\frak{g}_0)} \right), ~~ \wUg := K \otimes_{R} \widehat{U(\frak{g}_0)}.\]

Similarly for $\widehat{U(\frak{g}_0)}$ and $\wUg$,
we define $\wU^\circ_\lambda$, $\wU_\lambda$ to be the quotient rings by nullifying the relations $\{ (\Delta_i - \lambda_i) | 1 \leq i \leq r\}$.

For $\delta \in \wUg$, we say that $\delta$ is \emph{generic} if
the image of $\delta$ via $\wUg \twoheadrightarrow \wU_\lambda$ is non-zero for all $\lambda \in \ZZ_p^r$. 
An equivalent torsion condition is given by compactness of $\ZZ_p$ via the following lemma.

\begin{lem}\label{generic compact}
If $\delta \in \widehat{U(\frak{g}_0)}$ is generic, there exists a natural number $n_\delta \geq 1$ such that the image of $\delta$ via $\widehat{U(\frak{g}_0)} \twoheadrightarrow \wU_\lambda^\circ / p^{n_\delta}$ is non-zero for all $\lambda \in \ZZ_p^r$.
\end{lem}
\begin{proof}
We define a function $f_\delta : \ZZ_p^r \to \ZZ_{\geq 1}$:
for each $\lambda \in \ZZ_p^r$,
$f_\delta(\lambda)$ is defined to be the minimal positive integer such that the image of $\delta$ via $\widehat{U(\frak{g}_0)} \twoheadrightarrow \wU_\lambda^\circ / p^{f_\delta(\lambda)}$ is non-zero.
We have \[ \sum\limits_{i=1}^r \langle \Delta_i-\lambda_i \rangle + \langle p^{f_\delta(\lambda)} \rangle = \sum\limits_{i=1}^r \langle \Delta_i-(\lambda_i + p^{f_\delta(\lambda)} x_i) \rangle + \langle p^{f_\delta(\lambda)} \rangle ~ \for ~ \forall x_i \in \ZZ_p, \]
\[ \text{therefore} ~ f_\delta(\lambda+x) \leq f_\delta(\lambda) ~ \for ~ \forall x \in p^{f_\delta(\lambda)} \ZZ_p^r. \]
Since $\ZZ_p^r$ is compact and $f_\delta$ is upper-semicontinuous,
$f_\delta$ is bounded above.
\end{proof}

For $\bk = (k_1,\cdots,k_r) \in \NN^r$, we use $W_{\bk}^\circ$ (resp. $W_\bk$) to denote the $U(\frak{g}_0)$-module $\boxtimes_{i=1}^{r} \Sym^{k_i} R^2$ (resp. $\Ug$-module $\boxtimes_{i=1}^{r} \Sym^{k_i} K^2$).
We use $U_\bk$, $\wUk^\circ$, $\wUk$ for $U_\lambda$, $\wU^\circ_\lambda$, $\wU_\lambda$ when $\lambda = (\frac{1}{2}k_1(k_1+2),\cdots,\frac{1}{2}k_r(k_r+2)) $.

The goal of this section is to prove the following statement.
\begin{thm}\label{upper bound}
If $\whM$ is cyclic, torsion and a given generator is annihilated by a generic element $\delta \in \wUg$,
there exists a multiplicity-free polynomial $p_{\whM}$ in $r$ variables $k_1,\cdots,k_r$ of degree at most $r-1$ associated to $\whM$ such that for all $\bk \in \NN^r$, we have
\[ H^0_{\whM}(\bk) \leq p_{\whM}(\bk). \]
\end{thm}

To prove such a result, we will prove the image of $\widehat{U(\frak{g}_0)} \cdot \delta$ in $\End_{\QQ_p}(W_\bk)$ modulo $p^{n_\delta}$ ($n_\delta$ is a positive integer given by Lem \ref{generic compact}) requires suitably many generators over $\ZZ_p$.
This motivates us to estimate the number of generators using the filtration on $U(\frak{g}_0)$, which we illustrate below.
Since the natural $\widehat{U(\frak{g}_0)}$ action on $W_\bk$ factors through $\wUk^\circ$, our observations that $\wUk^\circ / p$ is an integral domain in Lem \ref{Ug quot cent} and this image is generated by the first $\bk$-th filtered piece of $U(\frak{g}_0)$ (Prop \ref{sl2 end}, Rem \ref{torsion fil image}) will achieve the desired estimate. 

To prevent any confusion, we emphasize that our argument is mostly integral. But we also include corresponding rational statements for completeness. 

Going back to the Lie algebra $\frak{g}_0$, for each $\frak{s}\frak{l}_{2,R}$-component, by the PBW theorem there is a $\ZZ$-filtration with $\mathrm{Fil}_l$ generated by polynomials of $\{h_i,e_i,f_i\}$ up to degree $l$.
The $\ZZ^r$ filtration is supported on $\NN^r \subset \ZZ^r$, so sometimes we write $\NN^r$ instead of $\ZZ^r$.

We equip $U(\frak{g}_0)$ with the product filtration indexed by $\Lambda$. 
The abelian group $\Lambda = \ZZ^r$ comes with the partial order \[ \lambda \leq \mu ~ \mathrm{if} ~ \lambda_1 \leq \mu_1, \cdots, \lambda_r \leq \mu_r. \]
We write \begin{eqnarray}\label{multi-index} 0 \ll \lambda ~\mathrm{if}~ 0 \ll \min\limits_{1 \leq i \leq r} \lambda_i, ~\mathrm{and}~ \lambda \to \infty ~\mathrm{if}~ \min\limits_{1 \leq i \leq r} \lambda_i \to \infty. \end{eqnarray}
We have $F_\mu U(\frak{g}_0) = 0$ if $\mu \notin \NN^r \subset \ZZ^r$.
And $\rank_R ~ F_\lambda U(\frak{g}_0) < \infty$ for all $\lambda \in \Lambda$, and there exists a polynomial $p_{U(\frak{g}_0)}$ in $r$ variables such that 
\[ \rank_R ~ F_\lambda U(\frak{g}_0) = p_{U(\frak{g}_0)}(\lambda) ~ \for ~ 0 \ll \lambda \in \NN^r. \]
But to form a graded ring, we use the $\ZZ$-filtration and its associated graded ring is \[ S_R := \gr(U(\frak{g}_0)) \simeq R[h_1,e_1,f_1,\cdots,h_r,e_r,f_r]. \]

Let $M$ be a finitely generated $U(\frak{g})$-module (resp. $U(\frak{g}_0)$-module), with a set of generators $\{m_1,\cdots,m_l\}$. 
We define two filtrations on $M$ valued in $\Lambda = \ZZ^r$ or $\Lambda = \ZZ$, both given by the formula (and similarly for $U(\frak{g}_0)$-modules) \[F_\lambda M := \sum\limits_{i=1}^l F_\lambda \Ug \cdot m_i, ~ \lambda \in \Lambda.\] 
If $\Lambda = \ZZ$, we have the associated graded module $\gr(M)$ over $S = \gr(U(\frak{g}))$.
Let $I_Z$ be the ideal of $S$ generated by $\{\frac{1}{2} h_i^2 + 2e_if_i ~ | ~ 1 \leq i \leq r\}$.

We prove some useful properties of the rings $U_\lambda$,$\wU^\circ_\lambda$ and $\wU_\lambda$:
\begin{lem}\label{Ug quot cent}
For $\lambda \in \ZZ_p^r$, $\frak{r} \in R$,
\begin{itemize}
\item Let $h,e,f$ be the basis of $\frak{s}\frak{l}_{2,R}$ with $\Delta=\frac{1}{2}h^2-h+2ef$ and the same commutation relations as (\ref{sl2 rel}), then $\{e^a f^b h^c | a, b \in \NN, c \in \{0,1\} \}$ is a basis of $U(\frak{s}\frak{l}_{2,R}) / (\Delta-\frak{r})$.
\item $F_\bullet U(\frak{s}\frak{l}_{2,R})$ induces a $\ZZ$-filtration on $U(\frak{s}\frak{l}_{2,R}) / (\Delta-\frak{r})$. 
Therefore both $\ZZ^r$ and $\ZZ$ filtrations induce corresponding filtrations on $U_\lambda^\circ$ and $U_\lambda$, and for the $\ZZ^r$ filtration, \[\rank_R ~ F_d U_\lambda^\circ = \dim_K F_d U_\lambda = \prod\limits_{i=1}^r (d_i + 1)^2 \] is a polynomial in $d = (d_1,\cdots,d_r) \in \NN^r$. 
\item The Noetherian ring $U_\lambda$ is an integral domain.

If $K$ is moreover a finite extension of $\QQ_p$ for $p \geq 3$:
\item The ring $\wU_\lambda^\circ$ is isomorphic to the $p$-adic completion of $U_\lambda^\circ$, and 
\[ K \otimes_R U_\lambda^\circ \simeq U_\lambda, ~~ K \otimes_R \wU_\lambda^\circ \simeq \wU_\lambda.\]
\item The rings $\wU_\lambda^\circ$, $\wU_\lambda$ are also Noetherian integral domains.
\end{itemize}
\end{lem}
\begin{proof}
The first part follows from the commutation relations and an induction on the total degree.

The second part follows from the first part.

The $\overline{K}$-algebra $\overline{K}[h,e,f]/(\frac{1}{2}h^2+2ef)$ is an integral domain over the algebraic closure $\overline{K}$ of $K$. And so is the tensor product of $r$-copies of $\overline{K}[h,e,f]/(\frac{1}{2}h^2+2ef)$.
Hence $S_K / I_Z$ is also an integral domain.
For the last part, consider the $\ZZ$-filtration on $U_\lambda$. 
Its graded ring $\gr(U_\lambda)$ is isomorphic to the integral domain $S_K / I_Z$ by Lem \ref{grade ring quotient generator}.
For any non-zero $a, b \in U_\lambda$, let $i_a$, $i_b$ be the minimal natural numbers such that $a \in F_{i_a} U_\lambda, b \in F_{i_b} U_\lambda$.
The image of $ab$ in \[F_{i_a+i_b} U_\lambda / F_{i_a+i_b-1} U_\lambda\] is non-zero, so is $ab \in U_\lambda$.

The fourth part follows from Lem \ref{ext comp} and the flat base change $R \hookrightarrow K$.

For the last part, we may assume $K = \QQ_p$.
We have $\wU_\lambda^\circ / p \simeq U_\lambda^\circ / p$ is isomorphic to $U_\lambda$ over $\FF_p$ (defined in (\ref{U lambda}) for $K = \FF_p$), the claim then follows from the third part for $\FF_p$.
\end{proof}

For the rest of the section, $K$ is a finite extension of $\QQ_p$ with the ring of integers $R$. 
Let $\Ann(\bk)$ (resp. $\widehat{\Ann(\bk)}$) be the annihilator ideal of $\Ug$ (resp. $\wUg$) for $W_{\bk}$.
As will be discussed in \S \ref{alg quo}, there is an algebraic isomorphism (\ref{Ann iso})
\[ \Ug / \Ann(\bk) \xrightarrow{\sim} \wUg / \widehat{\Ann(\bk)} \xrightarrow{\sim} \End_K (W_{\bk}). \]

If $\whM$ is a finitely generated $\wUg$-module, we define \[ \whM_\bk := \whM / \widehat{\Ann(\bk)} \cdot \whM \] and the $W_{\bk}$-\emph{multiplicity} $H^0_{\whM}(\bk)$ to be
\begin{eqnarray}\label{Wk mul} H^0_{\whM}(\bk) := \dim_K \Hom_{\wUg}(\whM, W_{\bk}) = \frac{\dim_K \whM_\bk}{\dim_K W_\bk}. \end{eqnarray}

Note that the natural $\Ug$, $\wUg$ actions on $W_\bk$ factor through $U_\bk$, $\wUk$.

\begin{prop}\label{sl2 end}
For any $\bk \in \NN^r$, as $R$-modules, the image of $F_\bk U_\bk^\circ$ equals to the full image of $U_\bk^\circ$ in $\End_R(W_\bk^\circ)$.
As a corollary, there are isomorphisms of $K$-vector spaces: 
\[ F_\bk U_\bk \xrightarrow{\sim} \Ug/\Ann(\bk) \xrightarrow{\sim} \End_K(W_\bk). \]
\end{prop}

\begin{proof}
It suffices to prove the case $r = 1$, $\bk = (k)$.
We first observe that for any $a \geq 0$, $e^a f^a$ can be expressed as a $R$-linear combination of $h^i$ as an operator in $U_\bk^\circ$.
For example, let $\lambda_k = \frac{1}{4} k(k+2)$, we have
\begin{eqnarray*}
e^2f^2 & = & e (\lambda_k+\frac{1}{2}h-\frac{1}{4}h^2) f \\
& = & \lambda_k \cdot ef + \frac{1}{2} (he-2e)f - \frac{1}{4} (he-2e)(fh-2f) \\
& = & (\lambda_k + h - 2)ef + \frac{1}{2} efh - \frac{1}{4} hefh \\
& = & (\lambda_k - 2 + \frac{3}{2} h - \frac{1}{4} h^2) (\lambda_k+\frac{1}{2}h-\frac{1}{4}h^2) 
\end{eqnarray*} 

Let $v_0,\cdots,v_k$ be the weight vectors of $W_k^\circ$ of increasing weights, with $W_\bk^\circ = \oplus_{i=0}^k R v_i$, $f$ raising a weight and $e$ lowering a weight.
Let $W_{\geq i} := \sum_{j \geq i}^k R v_j$.
Note that we want to prove each monomial $e^a f^b h^c$ ($c$ is either zero or one by the first part of Lem \ref{Ug quot cent}) with total degree strictly larger than $k$ can be reduced to a $R$-linear combination of lower degree terms.
For explanation, we only show the reduction for the monomial $e^a f^b$, where $a+b > k$, $a \leq b$, the monomial $e^a f^b = e^a f^a f^{b-a}$ (other cases can be similarly obtained). 
The operator $e^a f^a$ can be viewed as an endomorphism of $W_{\geq b-a}$ since $f^{b-a}$ maps $W$ to $W_{\geq b-a}$. 
We know it is a $R$-linear combination of powers of $h$, by Cayley-Hamilton,
\[ e^a f^a = \sum_{i=0}^{k-b+a} c_i h^i \in \End_R (W_{\geq b-a}), ~ c_i \in R. \]
As $k-b+a < 2a$, the degree of $e^a f^b$ is reduced.
We can similarly argue the other cases.

The corollary for $K$-isomorphisms follows from counting the dimensions of both sides by Lem \ref{Ug quot cent}.
\end{proof}

\begin{rem}\label{torsion fil image}
Let $E^\circ_\bk$ be the image of $U^\circ_\bk$ in $\End_R(W_\bk^\circ)$.
As it is $p$-adic complete, it coincides with the image of $\widehat{U(\frak{g}_0)}$ as well. 
Moreover we have \[ F_\bk( U^\circ_\bk / p^m ) \simeq E^\circ_\bk / p^m \] for all $m \geq 0$.
\end{rem}

\begin{proof}[Proof of Theorem \ref{upper bound}]
Let $\whM_0$ be a cyclic $\widehat{U(\frak{g}_0)}$-lattice inside $\whM$ such that $\whM_0 \otimes_R K \xrightarrow{\sim} \whM$.
And we may assume $\delta \in \widehat{U(\frak{g}_0)}$.
The surjection $\widehat{U(\frak{g}_0)} \twoheadrightarrow \whM_0$ corresponding to the generator factors through $\widehat{U(\frak{g}_0)} / \widehat{U(\frak{g}_0)} \delta$.

We pick $n_\delta \geq 1$ satisfying Lem \ref{generic compact}.
Under the natural identification $U(\frak{g}_0) / p^{n_\delta} \simeq \widehat{U(\frak{g}_0)} / p^{n_\delta}$,
there exists $\alpha \in \NN^r$ such that $\delta \in F_\alpha (\widehat{U(\frak{g}_0)} / p^{n_\delta}) \simeq F_\alpha (U(\frak{g}_0) / p^{n_\delta})$.

If $\bk \geq \alpha$, \[ F_{\bk-\alpha}(\wU_\bk^\circ / p^{n_\delta}) \cdot \delta \subset F_{\bk}(\wU_\bk^\circ / p^{n_\delta}). \]
Let $\varpi$ be a uniformizer of $R$.
For any $\bk \in \NN^r$, and $e \in U_\bk^\circ$ such that $e \notin \varpi U_\bk^\circ$, we have $e \cdot \delta \neq 0$ in $U_\bk^\circ / p^{n_\delta}$ since $U_\bk^\circ / \varpi$ is an integral domain by Lem \ref{Ug quot cent}.
The composition of maps of vector spaces
\[ F_{\bk-\alpha} U_\bk \to F_{\bk-\alpha} U_\bk \cdot \delta \to \End_K(W_\bk) \] is injective; otherwise, there exists $e \in F_{\bk-\alpha} U_\bk^\circ \backslash \varpi U_\bk^\circ$ such that $e \cdot \delta$ maps to \[0 \in E_\bk^\circ \hookrightarrow \End_K(W_\bk),\]
contradicting Rem \ref{torsion fil image} of Prop \ref{sl2 end} since $e \cdot \delta \in F_\bk( U^\circ_\bk / p^{n_\delta} )$ is non-zero modulo $p^{n_\delta}$.

Therefore the image of $\wUg \delta$ in $\wUg_\bk \simeq \End_K(W_\bk)$ has dimension at least $\dim_K F_{\bk-\alpha} U_\bk$, and we have the following bound for $(\wUg/\wUg \delta)_\bk$
\[ H^0_{\wUg/\wUg \delta}(\bk) = \frac{\dim_K (\wUg/\wUg \delta)_\bk}{\dim_K W_\bk} \leq \frac{\prod_{i=1}^r (k_i+1)^2-\prod_{i=1}^r (k_i-\alpha_i+1)^2}{\prod_{i=1}^r (k_i+1)}. \]
Regarding each $(k_i+1)$ as a variable, each term $c_{S, S'} \cdot \frac{\prod_{i \in S} (k_i+1)}{\prod_{i \in S'} (k_i+1)}$ is bounded by $|c_{S, S'}| \cdot \prod_{i \in S} (k_i+1)$ for $S \sqcup S' \subset \{1,\cdots,r\}$.
We get the desired bound for $\wUg/\wUg \delta$.
\end{proof}

\begin{rem}
If $M$ is a finitely generated module over $\Ug$, in general, the rank of $M$ does not have to agree with the rank of $U_\bk \otimes_{\Ug} M$ over $U_\bk$ without the genericity condition.
If $r=2, \frak{g} = \frak{s}\frak{l}_{2,K} \oplus \frak{s}\frak{l}_{2,K}$, and $\Delta_1, \Delta_2$ are Casimir operators for the two components,
the algebraic representations $W_{k,k}$ of parallel weights grow in quadratic order in the cyclic torsion $\Ug$-module $\Ug / (\Delta_1-\Delta_2)$.

These Casimir operators do not exist in $K[[G]]$, and we will see elements in $\wUg$ obtained from base change over the microlocalisation (\ref{micro}) are generic. 
\end{rem}

\section{Comparison of algebraic quotients}\label{alg quo} 
Let $K$ be a finite extension of $\QQ_p$ with the ring of integers $R$, residue field $k$.
Let $G$ be a uniform pro-$p$ group of dimension $d = \dim G$. 
We define the completed group rings by \[ R[[G]] := \lim _{\longleftarrow} R[G / N], ~ K[[G]] := K \otimes_R R[[G]], \]
where $N$ runs over all the open normal subgroups $N$ of $G$.

Lazard \cite{Laz65} defines a $\ZZ_p$-Lie algebra $L_G$ associated to $G$ (see also \cite[\S 4.5]{DDSMS99}).
We briefly recall some basic facts about $L_G$ here.
We fix a minimal topological generating set $\{g_1,\cdots,g_d\}$ of $G$. Each element of $G$ can be written uniquely in the form $g_1^{\lambda_1}\cdots g_d^{\lambda_d}$ for some $\lambda_1,\cdots,\lambda_d \in \ZZ_p$.
By \cite[Thm 4.30]{DDSMS99}, the operations
\begin{eqnarray}
\label{dot lie action} \lambda \cdot x &=& x^{\lambda} \\
x+y &=& \lim_{i \to \infty}\left(x^{p^{i}} y^{p^{i}}\right)^{p^{-i}} \\
\mathrm{[}x, y] &=& \lim_{i \to \infty}\left(x^{-p^{i}} y^{-p^{i}} x^{p^{i}} y^{p^{i}}\right)^{p^{-2 i}} 
\end{eqnarray}
define a Lie algebra structure $L_G$ on $G$ over $\ZZ_p$.
$L_G$ is a \emph{powerful Lie algebra} in the sense that it is free of rank $d = \dim G$ over
$\ZZ_p$ and satisfies $[L_G, L_G] \leq pL_G$.
Letting \[ \frak{g}_R = \frac{1}{p} L_G \otimes_{\ZZ_p} R, ~ \frak{g}_K = \frak{g}_R \otimes_{R} K, \] the \emph{completed universal enveloping algebras} $\widehat{U(\frak{g}_R)}$, $\widehat{U(\frak{g}_K)}$ are defined to be
\[ \widehat{U(\frak{g}_R)} :=\varprojlim_{a} \left(\frac{U(\mathfrak{g}_R)}{p^{a} U(\mathfrak{g}_R)}\right), ~ \widehat{U(\frak{g}_K)} := \widehat{U(\frak{g}_R)}\otimes_{R} K \]
following \cite{ST03} (appearing as the `largest' distribution algebra $D_{1/p}(G, K)$) and \cite{AW13}.

\begin{lem}
If $G$ is a compact open uniform pro-$p$ subgroup of $\SL_n(\QQ_p)$, then the associated Lie algebra \[ \frak{g}_K \simeq \frak{s}\frak{l}_{n,K} \] is isomorphic to the Lie algebra of $\frak{s}\frak{l}_n$ over $K$.
\end{lem}
\begin{proof}
This is an exercise \cite[Part II, Chapter 9, Ex 9]{DDSMS99} following from Lazard's paper \cite{Laz65}.
\end{proof}

As $G$ is a uniform pro-$p$-group, it is compact locally $\QQ_p$-analytic by \cite[Thm 8.18]{DDSMS99}. 
$G$ moreover satisfies the assumption (HYP) of \cite[\S 4]{ST03} by the remark before \cite[Lem 4.4]{ST03}.
Schneider and Teitelbaum have introduced the $K$-Fr\'echet-Stein algebra $D(G, K)$ of $K$-valued locally analytic distributions on $G$ (\cite{ST02J}, \cite{ST03}).
We briefly recall some basic properties of $D(G, K)$ from \cite{ST03} here.

Let $b_i := g_i - 1 \in R[G]$, and write \begin{eqnarray}\label{balpha} \mathbf{b}^{\alpha}=b_{1}^{\alpha_{1}} \cdots b_{d}^{\alpha_{d}} \in R[G] \end{eqnarray} for any $d$-tuple $\alpha \in \NN^d$.
We write $|\alpha| := \sum_{i=1}^d \alpha_i$.
It follows from the proof of \cite[Thm 7.20]{DDSMS99} that $R[[G]]$ can be naturally identified with the set of non-commutative formal power series in $b_1,\cdots, b_d$ with coefficients in $R$:
\[ R [[G]] =\left\{\sum_{\alpha \in \mathbb{N}^{d}} \lambda_{\alpha} \mathbf{b}^{\alpha} \mid \lambda_{\alpha} \in R\right\}. \]

There is a faithfully flat natural map from the Iwasawa algebra to the distribution algebra \begin{eqnarray}\label{ST dis} K[[G]] \to D(G, K) \end{eqnarray} by \cite[Thm 4.11]{ST03}, such that $D(G, K)$ can be identified with power series in $b_1,\cdots, b_d$ with convergence conditions
\[ D(G, K) = \left\{\sum_{\alpha \in \mathbb{N}^{d}} \lambda_{\alpha} \mathbf{b}^{\alpha} \mid \lambda_\alpha \in K, ~ \mathrm{and} ~ \for ~ \forall ~ 0 < r < 1, ~ \sup_{\alpha \in \mathbb{N}^{d}} |\lambda_\alpha| r^{|\alpha|} < \infty \right\}. \]

For $G$, there is an integrally valued \emph{$p$-valuation}
$\omega : G \backslash \{ 1 \} \to \ZZ_{\geq 1}$ such that
\begin{eqnarray*}  \omega(g h^{-1}) & \geq & \min (\omega(g), \omega(h)) \\ \omega (g^{-1} h^{-1} g h) & \geq & \omega(g)+\omega(h), \text{and} \\  \omega (g^{p}) & = & \omega(g)+1 \end{eqnarray*}
for any $g, h \in G$, with $\omega(1) := \infty$ \cite[III 2.1.2]{Laz65}.

For $G$ being uniform, we may define $\omega(g)$ to be $n \geq 1$ such that $g \in G^{p^{n-1}} \backslash G^{p^n}$.
It is indeed an integrally valued $p$-valuation on $G$ by for example \cite[Lem 10.2]{AW13}.
By the discussion in \cite[\S 4.2]{DDSMS99}, $\omega$ is characterized by 
\begin{eqnarray}\label{all val one} \omega(g_i) = 1 ~ \for ~ 1 \leq i \leq d, ~ \text{and} \end{eqnarray}
\[ \omega(g) = 1+\min_{1 \leq i \leq d} \omega_p(x_i), ~ \for ~ \forall g = g_1^{x_1}\cdots g_d^{x_d} \in G, \] where $\omega_p$ denotes the $p$-adic valuation on $\ZZ_p$.

The Fr\'echet topology of $D(G, K)$ is defined by the family of norms
\[ \|\lambda\|_{r}:=\sup _{\alpha \in \mathbb{N}^{d}}\left|\lambda_{\alpha}\right| r^{|\alpha|} \] for $0 < r < 1$, where the absolute value $|\cdot|$ is normalized as usual by $|p| = p^{-1}$. We let
\[ D_r(G, K) := \text{completion of $D(G, K)$ with respect to the norm} ~ \| ~ \|_r. \]
As a $K$-Banach space,
\begin{eqnarray}\label{Dr series}
 D_r(G, K) = \left\{\sum_{\alpha \in \mathbb{N}^{d}} \lambda_{\alpha} \mathbf{b}^{\alpha} \mid \lambda_\alpha \in K, ~ \sup_{\alpha \in \mathbb{N}^{d}} |\lambda_\alpha| r^{|\alpha|} < \infty \right\}. 
\end{eqnarray}

\begin{thm}
If $1/p \leq r < 1$, $D_r(G, K)$ is a Banach noetherian integral domain with multiplicative norm $\| ~ \|_r$. The distribution algebra 
\[ D(G, K) = \varprojlim_{r} D_r(G, K) \] is a $K$-Fr\'echet-Stein algebra.
\end{thm}
\begin{proof}
This is the main result of \cite[\S 4]{ST03}.
\end{proof}
We remark that Schneider-Teitelbaum's definition of $\| ~ \|_r$ is slightly more complicated in general, but agrees with our $\| ~ \|_r$ because of (\ref{all val one}) due to the uniform assumption of $G$.

Let $\frak{m} := \ker(R[[G]] \to k)$ be the unique maximal ideal of $R[[G]]$.
Following \cite[\S 10]{AW13}, we consider a microlocal Ore set $S_0$:
\begin{eqnarray}\label{S0} S_0 := \bigcup_{a \geqslant 0}\left(p^{a}+\mathfrak{m}^{a+1}\right) \subseteq R [[G]].\end{eqnarray}

Associated to $S_0$, there is a flat extension (see remarks in \cite[\S 1.4]{AW13}) \begin{eqnarray}\label{micro} K[[G]] \to \wUgK \end{eqnarray} by construction of \cite[\S 10]{AW13}, called the \emph{microlocalisation} of Iwasawa algebra.
For a finitely generated $K[[G]]$-module $\wtM$, we use $\whM := \wUgK \otimes_{K[[G]]} \wtM$ to denote the \emph{microlocalisation} of $\wtM$.

From now on we assume $\frak{g}_K$ is a split, semisimple Lie algebra over $K$. 
We refer to \cite{Bou05} for a treatment of such a Lie theory.
Let $W$ be an irreducible finite dimensional representation of $U(\frak{g}_K)$ with a $U(\frak{g}_R)$-lattice $W_0 \subset W$.

$W_0$ is of finite rank, therefore automatically $p$-adic complete.
The $U(\frak{g}_R)$ action on $W_0$ extends to a $\widehat{U(\frak{g}_R)}$ action,
and the $U(\frak{g}_K)$ action on $W$ extends to a $\widehat{U(\frak{g}_K)}$ action (also see \cite[\S 9.2]{AW13}).

We pull back $W$ as a $K[[G]]$-module via the microlocalisation (\ref{micro}).
Iwasawa modules arising from this way are called \emph{Lie modules} in \cite[\S 11.1]{AW13}.
By \cite[Thm 11.1, Cor 11.1]{AW13}, $W$ remains irreducible as a $K[[G]]$-module.

Let \[ \Ann(W) \subset U(\frak{g}_K), ~ \widehat{\Ann(W)} \subset \wUgK, ~ \widetilde{\Ann(W)} \subset K[[G]] \] respectively be the annihilator ideals of $U(\frak{g}_K), \wUgK, K[[G]]$ for $W$.
Similarly, we have
\begin{equation}\label{Ann iso}
U(\frak{g}_K) / \Ann(W) \xrightarrow{\sim} \wUgK / \widehat{\Ann(W)} \xrightarrow{\sim} \End_K (W), 
\end{equation}
\begin{equation}\label{KG iso}
K[[G]] / \widetilde{\Ann(W)} \hookrightarrow \wUgK / \widehat{\Ann(W)} \xrightarrow{\sim} \End_K (W). 
\end{equation}
The second map of (\ref{Ann iso}) is surjective by \cite[Chapter VIII, \S 6.2, Cor of Prop 3]{Bou05}.  

By \cite[Thm 11.1]{AW13}, every finite dimensional $K[[G]]$-module is semisimple.
For a finitely generated $K[[G]]$-module $\wtM$, we use \[ \wtM_W := \wtM / \widetilde{\Ann(W)} \cdot \wtM ~~ (\mathrm{resp.} ~ \whM_W := \whM / \widehat{\Ann(W)} \cdot \whM) \] to denote the maximal quotient of $\wtM$ (resp. $\whM$) which is isomorphic to a finite sum of $W$ as a $K[[G]]$-module (resp. $\wUgK$-module).

\begin{thm}\label{KG quo}
Let $\wtM$ be a finitely generated $K[[G]]$-module with microlocalisation $\whM = \wUgK \otimes_{K[[G]]} \wtM$. Then the natural map \[ \wtM_W \xrightarrow{\sim} \whM_W \] is an isomorphism.
In particular, $K[[G]] / \widetilde{\Ann(W)} \xrightarrow{\sim} \wUgK / \widehat{\Ann(W)}$ in (\ref{KG iso}) is an isomorphism.
\end{thm}
\begin{proof}
We apply the flat base change microlocalisation $K[[G]] \to \wUgK$ (\ref{micro}) to the short exact sequence \[ 0 \to \widetilde{\Ann(W)}\cdot \wtM \to \wtM \to \wtM_W \to 0 \] to get 
\begin{eqnarray}\label{micro exact} 0 \to \wUgK \otimes_{K[[G]]} \widetilde{\Ann(W)}\cdot \wtM \to \whM \to \wUgK \otimes_{K[[G]]} \wtM_W \to 0. \end{eqnarray}
The microlocal Ore set $S_0$ (\ref{S0}) acts invertibly on $W$ as $S_0$ consists of units in $\wUgK$ ($S_0$ is inverted to form the microlocalisation in Ardakov-Wadsley's construction, also see the proof of part (b) of \cite[Thm 11.1]{AW13}).
We can apply \cite[Prop 11.1]{AW13} to $W$ taking $n = 0$, which asserts that
the natural map $\wtM_W \xrightarrow{\sim} \wUgK \otimes_{K[[G]]} \wtM_W$ is an isomorphism as $K[[G]]$-modules because $\wtM_W$ is isomorphic to a direct sum of $W$.
By our assumption on $W$, the $K[[G]]$ action on $\wtM_W$ (uniquely) extends to $\wUgK$, it is moreover an isomorphism over $\wUgK$ because of the natural $\wUgK$-equivariant reverse map $\wUgK \otimes_{K[[G]]} \wtM_W \to \wtM_W$.

By maximality of $\whM_W$, the exact sequence (\ref{micro exact}) gives \[ \widehat{\Ann(W)} \cdot \whM \subset \wUgK \otimes_{K[[G]]} \widetilde{\Ann(W)}\cdot \wtM, \]
both as submodules of $\whM$.
From (\ref{KG iso}) we get $\wUgK \cdot \widetilde{\Ann(W)} \subset \widehat{\Ann(W)}$, therefore \[ \wUgK \otimes_{K[[G]]} \widetilde{\Ann(W)}\cdot \wtM \subset \widehat{\Ann(W)} \cdot \wUgK \otimes_{K[[G]]} \wtM = \widehat{\Ann(W)} \cdot \whM. \]
This forces $\wUgK \otimes_{K[[G]]} \widetilde{\Ann(W)} \cdot \wtM = \widehat{\Ann(W)} \cdot \whM$ and therefore $\whM_W = \wtM_W$.
\end{proof}

\section{Infinitesimal specialization}\label{specialization}

We continue to use notation from \S \ref{alg quo}.
For simplicity, we take $K$ to be $\QQ_p$.
The norm $| ~ |$ on $\QQ_p$ is normalized as usual by $| p | = p^{-1}$.
Let $C(G, \QQ_p)$ and $C^\la(G, \QQ_p)$, $C^{\text{sm}}(G, \QQ_p)$ be respectively the space of continuous functions on $G$, the space of locally analytic functions on $G$, and the space of smooth functions on $G$, all valued in $\QQ_p$.
The Lie algebra $\frak{g}$ acts on $C^\la(G, \QQ_p)$ by continuous endomorphisms defined by \begin{eqnarray}\label{lie la} \mathfrak{x} f:=\lim_{t \rightarrow 0} \frac{ (t \cdot \mathfrak{x}) f-f}{t} \end{eqnarray}
for $\frak{x} \in \frak{g}_0, f \in C^\la(G, \QQ_p)$ where the dot action $t \cdot \frak{x}$ is given by (\ref{dot lie action}).
We have a natural inclusion 
\begin{eqnarray}\label{inc lie to dist}
U(\frak{g}) \to D(G,\QQ_p). \end{eqnarray}

We further assume that $G$ is an open subgroup of the group of $\QQ_p$-rational points of a connected split reductive $\QQ_p$-group $\GG$ with Borel pair $(\BB,\TT)$. 
The $\QQ_p$-split Lie algebra of $\GG$ should be identified with the Lie algebra $\frak{g}$ associated to $G$ in \S \ref{alg quo}. 
We use $\frak{g}_0$ to denote $\frak{g}_{\ZZ_p}$.

Let $T := \TT(\QQ_p) \cap G$ be a torus of $G$, $B := \BB(\QQ_p) \cap G$ be a Borel subgroup with unipotent radical $N$ such that $B = TN$.
Let $\frak{t}_0$, $\frak{b}_0$, $\frak{n}_0$ respectively be their associated Lie algebras over $\ZZ_p$, with generic fibres $\frak{t}$, $\frak{b}$, $\frak{n}$.
Suppose $\overline{\frak{n}}_0$ is an opposite nilpotent of $\frak{n}$ such that $\frak{g}_0$ admits a triangular decomposition
\[ \frak{g}_0 = \overline{\frak{n}}_0 \oplus \frak{t}_0 \oplus \frak{n}_0. \]
Let $\Zgo$, $Z(\frak{g})$ respectively be the centers of $U(\frak{g}_0)$, $U(\frak{g})$.
Any character of $\Zgo$ is naturally extended to a character of $\Zg$.
There is the Harish-Chandra homomorphism \cite[Chap VIII, \S 6.4]{Bou05} associated to the triangular decomposition
\begin{eqnarray}\label{HC hom}
\mathrm{HC}: Z(\frak{g}) \to U(\frak{t}).
\end{eqnarray}

If $\chi : T \to 1+p\ZZ_p$ is a continuous/locally analytic character of the torus $T$, it induces a character $d\chi: U(\frak{t}_0) \to \ZZ_p$ by formula (\ref{lie la}), extending to $d\chi: U(\frak{t}) \to \QQ_p$.
We call an infinitesimal character $\lambda: Z(\frak{g}) \to \QQ_p$ \emph{induced} if $\lambda = d\chi \circ \mathrm{HC}$ for a character $\chi$ of $T$.

Our main theorem in this section is the following.
\begin{thm}\label{mic inf sp}
Let $\lambda$ be an induced infinitesimal character.
If $G$ is a finite product of first congruence subgroups of $\SL_2(\ZZ_p)$,
the composition of microlocalisation (\ref{micro}) with infinitesimal specialization is injective.
\[ \QQ_p[[G]] \hookrightarrow \widehat{U(\frak{g})} \otimes_{\Zg,\lambda} \QQ_p. \]
\end{thm}

The corresponding statement for distribution algebra turns out to be much easier and it serves as a first step to prove Thm \ref{mic inf sp}.

For a proof, we make use of the locally analytic principal series of $\chi : T \to 1+p\ZZ_p$
\[ \Ind(\chi) := \{f: G \to \QQ_p | f ~ \text{locally analytic}, ~ f(gtn)= \chi(t) f(g), ~ \forall t \in T, n \in N, g\in G \}. \]
The locally analytic principal series has an induced infinitesimal character determined by $d \chi$ via the Harish-Chandra homomorphism (\ref{HC hom}).

\begin{thm}\label{dis inf sp}
As $\QQ_p$-Fr\'echet spaces, $D(G,\QQ_p) \otimes_{\Zg,\lambda} \QQ_p$ is the strong dual of the locally convex vector space of compact type $C^\la(G, \QQ_p)[\lambda]$, where $C^\la(G, \QQ_p)[\lambda]$ is the $\lambda$-isotypic part of the space of locally analytic functions on $G$.

If the infinitesimal character $\lambda$ is induced,
the composition of (\ref{ST dis}) with infinitesimal specialization is injective.
\[ \QQ_p[[G]] \hookrightarrow D(G,\QQ_p) \otimes_{\Zg,\lambda} \QQ_p. \] 
\end{thm}
\begin{proof}
Let $z_1,\cdots,z_n$ be a set of generators of the kernel of $\lambda$. 
By \cite[Prop 3.7]{ST02J}, these differential operators define $G$-equivariant endomorphisms of $C^\la(G, \QQ_p)$.
Therefore we have a short left exact sequence of admissible locally analytic representations of $G$:
\begin{eqnarray}
\label{left short la} 0 \to C^\la(G, \QQ_p)[\lambda] \to C^\la(G, \QQ_p) & \to & C^\la(G, \QQ_p)^n \\
f & \mapsto & (z_i \cdot f)_i.
\end{eqnarray}
By taking the strong dual of (\ref{left short la}) and the anti-equivalence of categories between admissible locally analytic representations and coadmissible $D(G, \QQ_p)$-modules \cite[Thm 6.3]{ST03}, we identify $D(G, \QQ_p) \otimes_{\Zg,\lambda} \QQ_p$ with $C^\la(G, \QQ_p)[\lambda]'_b$.

For the second part, as $\QQ_p[[G]]$ is dual to $C(G, \QQ_p)$, it suffices to prove $C^\la(G, \QQ_p)[\lambda]$ is dense in $C(G, \QQ_p)$.

We choose any locally analytic character $\chi : T \to K^\times$ such that $\Ind(\chi)$ has the infinitesimal character $\lambda$.
Let $X$ be the quotient space $X := G / B$, then there is a splitting $X \hookrightarrow G \twoheadrightarrow X$ as $p$-adic manifolds such that 
$\Ind(\chi) \simeq C^\la(X, \QQ_p)$ as topological $K$-vector spaces.
We choose any nowhere vanishing function $f_0 \in C^\la(X, \QQ_p)$ so that $f_0 \in C^\la(G, \QQ_p)[\lambda]$.
The pointwise product of $f_0$ with any smooth function still has the infinitesimal character $\lambda$.
We see that $f_0 \cdot C^{\text{sm}}(G, \QQ_p) \subset C^\la(G, \QQ_p)[\lambda] \subset C(G, \QQ_p)$ is clearly dense in $C(G, \QQ_p)$.
\end{proof}

Following \cite{Fro03}, \cite{Koh07}, we define $U_r(\frak{g})$ to be the closure of $U(\frak{g})$ in $D_r(G, \QQ_p)$ with respect to the norm $\| ~ \|_r$ for $0 < r < 1$.
If $\lambda : \Zg \to \QQ_p$ is an infinitesimal character, we define
\[ U_r^\lambda(\frak{g}) := U_r(\frak{g}) \otimes_{\Zg,\lambda} \QQ_p, ~ D_r^\lambda(G, \QQ_p) := D_r(G, \QQ_p) \otimes_{\Zg,\lambda} \QQ_p. \]

\begin{prop}\label{crossed prod Dr}
If $r = \sqrt[p^n]{1 / p}$ for $n \in \ZZ_{\geq 1}$, $D_r(G, \QQ_p)$ is a crossed product (\cite[\S 1.5.8]{MR01}, \cite[Proof of Prop 10.6]{AW13}) of $U_r(\frak{g})$ by $G / G^{p^n}$.
Consequently, \[ D_r^\lambda(G, \QQ_p) \simeq U_r^\lambda(\frak{g}) * (G / G^{p^n}) \] is a crossed product of $U_r^\lambda(\frak{g})$ by $G / G^{p^n}$.
\end{prop}
\begin{proof}
See \cite[(6.8), Cor 5.13]{Sch13} for a similar statement of the first claim. 
We give a proof as follows.
By (\ref{lie la}), for any $g \in G, \frak{x} \in \frak{g}_0, f \in C^\la(G, \QQ_p)$,
\begin{eqnarray*} g (\frak{x} \cdot (g^{-1} f)) & = & g \left( \lim_{t \rightarrow 0} \frac{ (t \cdot \mathfrak{x}) (g^{-1}f)-(g^{-1}f)}{t} \right) \\ & = & \lim_{t \rightarrow 0} \frac{ g (t \cdot \mathfrak{x}) g^{-1}f-f}{t} \\ & = & \lim_{t \rightarrow 0} \frac{ (t \cdot g \mathfrak{x} g^{-1})f-f}{t} \\ & = & (g \mathfrak{x} g^{-1}) \cdot f. \end{eqnarray*}
We see $\Ug$ is stable under the conjugation action of $G$.
By Frommer's theorem \cite{Fro03}, \cite[Thm (Frommer), Proof of Cor 1.4.1]{Koh07}, \begin{eqnarray}\label{Ur conver rate} U_{r}(\mathfrak{g})=\left\{\sum_{\alpha} d_{\alpha} \mathfrak{X}^{\alpha} | d_{\alpha} \in \QQ_p, \lim_{|\alpha| \rightarrow \infty} | d_{\alpha} | \left\|\mathfrak{X}^{\alpha}\right\|_{r}=0\right\} \end{eqnarray}
for $\mathfrak{X}^{\alpha} = \log(1+b_1)^{\alpha_1}\cdots \log(1+b_d)^{\alpha_d}$ compared to (\ref{balpha}),
and $D_r(G, \QQ_p)$ is a free (left or right) $U_r(\frak{g})$-module of basis given by representatives of $G / G^{p^n}$.
It suffices to prove for any $g \in G$, its $p^n$-th power belongs to $U_r(\frak{g})$, and we may further assume $g = b_i$ for certain $1 \leq i \leq d$.
As a formal power series, we have 
\[ g^{p^n} = \exp(p^n \frak{X}_i ) = \sum_{k=0}^\infty \frac{p^{nk}\frak{X}_i^k}{k!}, \] for $\frak{X}_i := \log(1+b_i)$.
Here by Taylor series of $\log(1+x)$ we have \[ \| \frak{X}_i \|_r = \max_{i \geq 1} \frac{r^i}{|i|} = \max_{k \geq 0} \frac{r^{p^k}}{|p^k|} = \max_{k \geq 0} p^{-\frac{p^k}{p^n}} \cdot p^k = \max_{k \geq 0} p^{(k-\frac{p^k}{p^n})} = p^{n-1}, \] and since $\| ~ \|_r$ is multiplicative by \cite[Thm 4.5]{ST03}, $\| \frak{X}_i^k \|_r = p^{(n-1)k}$, we have 
\[ \lim_{k \to \infty} |\frac{p^{nk}}{k!}| \cdot p^{(n-1)k} \leq p^{\frac{k}{p-1}} \cdot p^{-k} = 0. \]
As a consequence $g^{p^n} \in U_r(\frak{g})$ by characterization (\ref{Ur conver rate}).

The second claim follows from the first claim since $\Zg$ is in the center of $D_r(G, \QQ_p)$ by \cite[Prop 3.7]{ST02J} and $\Zg \subset U(\frak{g}) \subset U_r(\frak{g})$.
\end{proof}

\begin{rem}
\cite[Rem 10.6]{AW13} points out $D_r(G,\QQ_p)$ should be a crossed product of the microlocalisation of $\QQ_p[[G^{p^n}]]$ by $G / G^{p^n}$ for $r = \sqrt[p^n]{1 / p}$.
It is quite likely that $U_r(\frak{g})$ coincides with such a microlocalisation.
\end{rem}

For each $n \geq 0$, $G^{p^n}$ is isomorphic to its (unnormalized) Lie algebra \[ L_{G^{p^n}} \simeq p^{n+1}\ZZ_p^d \simeq \ZZ_p^d \] as $p$-adic manifolds. 
We define \[ C^{n,\an}(G, \QQ_p) := \left\{ f \in C(G, \QQ_p) \mid  f ~ \text{is analytic on each} ~ G^{p^n} ~ \text{coset} \right\}, \] 
\begin{eqnarray}\label{la ind an} \text{with} ~ C^\la(G, \QQ_p) = \varinjlim_{n \geq 0} C^{n,\an}(G, \QQ_p). \end{eqnarray}
In particular, $C^\an(G, \QQ_p) := C^{0,\an}(G, \QQ_p)$ is the space of analytic functions on $G$.

For any $n \geq 0$, let $r_n := \sqrt[p^n]{1 / p}$, the transition of spaces of analytic functions of decreasing radius is compact
\[ C^{n,\an}(G, \QQ_p) \hookrightarrow C^{n+1,\an}(G, \QQ_p). \]
For any $n \geq 1$, if $z_1^{n-1},\cdots,z_d^{n-1}$ are coordinates of $G^{p^{n-1}}$, 
then $z_1^n = pz_1^{n-1},\cdots, z_d^n = pz_d^{n-1}$ are coordinates of $G^{p^n}$.
For any $g \in G^{p^{n+1}}$, we may pullback \[ g^\ast (z_i^{n-1}) = \sum_{\alpha \in \ZZ_{\geq 0}^d} c_\alpha^i (z^{n-1})^\alpha \in C^\an(G^{p^{n-1}}, \QQ_p), ~ c_\alpha^i \in \ZZ_p, \]
and the constant term $c^i_{\underline{0}}$ is divided by $p^2$.
Consider the commutative diagram
\[ \xymatrix{
C^{n-1,\an}(G, \QQ_p) \ar[r] \ar[d]^{g^\ast} & C^{n,\an}(G, \QQ_p) \ar[d]^{g^\ast}\\
C^{n-1,\an}(G, \QQ_p) \ar[r] & C^{n,\an}(G, \QQ_p),\\
}\]
we have 
\[ g^\ast (z_i^n) = \sum_{\alpha \in \ZZ_{\geq 0}^d} p^{|\alpha|-1} c_\alpha^i (z^n)^\alpha \in C^\an(G^{p^n}, \QQ_p), ~ c_\alpha^i \in \ZZ_p. \]
Since $g \in G^{p^{n+1}}$ induces the identical map on $G^{p^n}/G^{p^{n+1}}$, we have $c_j^i \equiv \delta_{ij}$ mod $p$ for $z_j^n$'s coefficient $c_j^i$ in $g^\ast (z_i^n)$.
We see that the operator $g-1$ has norm at most $1/p$ on the Banach space $C^{n,\an}(G, \QQ_p)$ for any $n \geq 1$, $g \in G^{p^{n+1}}$,
and so is the operator $(g-1)^{p^{n+1}}$ for any $g \in G$.
Any power series of $D_r(G, \QQ_p)$ converges as an endomorphism of $C^{n,\an}(G, \QQ_p)$
\[ D_{r_{n+1}}(G, \QQ_p) \to \End_{\QQ_p}(C^{n,\an}(G, \QQ_p)), \]
by the description (\ref{Dr series}).
Composed with the evaluation map at identity \begin{eqnarray*} C^{n,\an}(G, \QQ_p) & \xrightarrow{\text{ev}} & \QQ_p \\
f & \mapsto & f(\text{id}), \end{eqnarray*}
we have a natural map
\begin{eqnarray}\label{dual n an} D_{r_{n+1}}(G, \QQ_p) \to (C^{n,\an}(G, \QQ_p))'_b. \end{eqnarray}

\begin{prop}\label{D inv lim iso}
There is an algebraic isomorphism
\[ D(G, \QQ_p) \otimes_{\Zg,\lambda} \QQ_p \xrightarrow{\sim} \varprojlim_{n \geq 1} D_{r_n}^\lambda(G, \QQ_p) . \]
\end{prop}
\begin{proof}
It suffices to construct the inverse of the natural map
\[ D(G, \QQ_p) \otimes_{\Zg,\lambda} \QQ_p \to \varprojlim_{n \geq 1} (D_{r_n}(G, \QQ_p) \otimes_{\Zg,\lambda} \QQ_p). \]
By Thm \ref{dis inf sp}, \[ D(G, \QQ_p) \otimes_{\Zg,\lambda} \QQ_p \to (C^\la(G, \QQ_p)[\lambda] )'_b. \]
From (\ref{la ind an}), 
\[ (C^\la(G, \QQ_p)[\lambda])'_b \simeq \varprojlim_{n \geq 0} (C^{n,\an}(G, \QQ_p)[\lambda])'_b \] as $\QQ_p$-Fr\'echet spaces.
The inverse \[ \varprojlim_{n \geq 0} D_{r_{n+1}}(G, \QQ_p) \otimes_{\Zg,\lambda} \QQ_p \to \varprojlim_{n \geq 0} (C^{n,\an}(G, \QQ_p)[\lambda])'_b \] is given by (\ref{dual n an}).
\end{proof}

Let $m \geq 1$ be a positive integer.
For each $1 \leq i \leq m$, let \[ G_i := (I_2 + p \mathrm{M}_2(\ZZ_p)) \cap \SL_2(\ZZ_p). \]
For the rest of this section, let $G = \prod\limits_{i=1}^m G_i \subset \widetilde{G} := \prod\limits_{i=1}^m \widetilde{G}_i$ be $r$ copies of the first congruence subgroup of $\SL_2(\ZZ_p)$. 

\begin{lem}\label{sl2 lie}
The associated Lie algebra for $G_1=\operatorname{ker} (\SL_2 (\ZZ_p) \rightarrow \mathrm{SL}_{2} (\mathbb{F}_{p} ))$ \[ \frak{g}_0 \simeq \frak{s}\frak{l}_{2,\ZZ_p} = \ZZ_p \cdot h \oplus \ZZ_p \cdot e \oplus \ZZ_p \cdot f,\] has the standard commutator brackets \[ [h, e]=2 e, \quad[h, f] = -2 f, \quad[e, f]=h .\]
\end{lem}
\begin{proof}
We set \[ h = \begin{pmatrix} 1 & 0 \\ 0 & -1 \end{pmatrix}, ~ e = \begin{pmatrix} 0 & 1 \\ 0 & 0 \end{pmatrix}, ~ f = \begin{pmatrix} 0 & 0 \\ 1 & 0 \end{pmatrix} \in \mathrm{M}_2(\ZZ_p). \]
And let $\exp(ph), \exp(pe), \exp(pf)$ be a set of minimal generators of $G_1$.
The $p$-adic manifolds $G_1$ and $p\frak{s}\frak{l}_2(\ZZ_p)$ are identified via the exponential and logarithm maps, the commutator brackets on $p\frak{s}\frak{l}_2(\ZZ_p)$ transfer to $G_1$ for example by the computations in \cite[Lem 7.12]{DDSMS99}.
\end{proof}

For example if $m=1$, the Iwasawa algebra $\ZZ_p[[G]]$ is identified with a non-commutative formal power series ring $\ZZ_p[[F,H,E]]$ in three variables for \[ F:= \exp(pf) - 1, ~ H := \exp(ph) - 1, ~ E := \exp(pe) - 1 \] as in the lemma.
Actually we may explicitly describe the microlocalisation map (\ref{micro}) for our case when we identify $p\frak{s}\frak{l}_2(\ZZ_p)$ with $L_G = p \frak{g}_0$ (proof of \cite[Thm 10.4]{AW13}):
\[ F \mapsto \exp(pf) - 1, ~ H \mapsto \exp(ph) - 1, ~ E \mapsto \exp(pe) - 1. \]
Under such an identification, the Lie algebra action (\ref{lie la}) is equivalent to 
\begin{eqnarray}\label{lie alg log} \mathfrak{x} f:= \log(1+ (\mathfrak{x}-1) ) \cdot f \end{eqnarray}
for $\mathfrak{x} \in \frak{g}_0 \simeq G$ and $f \in C^\la(G, \QQ_p)$.

\begin{thm}\label{Dr domain}
Let $p$ be an odd prime.
If $1/p < r < 1$, $D_r^\lambda(G, \QQ_p)$ is an integral domain.
\end{thm}
\begin{proof}
For saving notation, we prove for the case $m = 1$.
Let $F,H,E$ be the formal variables with \[ pf = \log(1+F), ~ ph = \log(1+H), ~ pe = \log(1+E) \] defined after Lem \ref{sl2 lie}.

For any $1/p < r < 1$, we can find $r' \leq r$ such that $1/p < r' < \sqrt{1 / p}$.
The sequence $\{ \frac{(r')^k}{| k |} | k \in \ZZ_{\geq 1} \}$ has the property that
\[ \max_{k \geq 1} \frac{(r')^k}{| k |} = \max_{i \geq 0} \frac{(r')^{p^i}}{| p^i |} = r'. \]
Note that we can endow the topology of $D_{r'}(G, \QQ_p)$ on $D_r(G, \QQ_p)$ as \begin{eqnarray}\label{r rprime} D_r(G, \QQ_p) \hookrightarrow D_{r'}(G, \QQ_p) \end{eqnarray} is naturally a dense subalgebra by characterization (\ref{Dr series}).

On $R = D_r(G, \QQ_p)$, or $D_{r'}(G, \QQ_p)$, we have the filtration 
\[ F_{r'}^{s} R:=\left\{ a \in R:\|a\|_{r'} \leq p^{-s} \right\}. \]
The associated graded ring is denoted by $\gr_{r'}(R)$.
By the density of (\ref{r rprime}), we have the isomorphism \[ \gr_{r'} D_r(G, \QQ_p) \simeq \gr_{r'} D_{r'}(G, \QQ_p) \simeq \mathbb{F}_{p}[\epsilon,\epsilon^{-1}][F,H,E] \] by \cite[Thm 4.5]{ST03}.

Let $\Delta := \frac{1}{2} h^2-h+2e f$ be the Casimir operator.
The kernel of $\lambda$ is generated by $p^2 \Delta+\lambda_0$ for $\lambda_0 \in \QQ_p$.
The associated graded ring for $D_r^\lambda(G, \QQ_p)$ is 
\[ \gr_{r'} D_r^\lambda(G, \QQ_p) \simeq \gr_{r'} D_r(G, \QQ_p) / \gr_{r'}(p^2 \Delta+\lambda_0) \] by Lem \ref{grade ring quotient generator}, where
the principal symbol of generator $\gr_{r'}(p^2 \Delta+\lambda_0)$ (\S \ref{rank}) in $\gr_{r'}(D_r(G, \QQ_p))$ equals to 
\[ \gr_{r'}(\frac{1}{2}\log(1+H)^2 - p\log(1+H) + 2\log(1+E)\log(1+F) + \lambda_0) = \gr_{r'}(\frac{1}{2}H^2+2EF+\lambda_0) \] by the assumption of $r'$.
If the valuation of $\lambda_0$ is at most $1$, \[ \gr_{r'}(p^2 \Delta+\lambda_0) = \gr_{r'}(\lambda_0) \] is a unit in $\gr_{r'} D_r(G, \QQ_p)$, making the quotient equal to zero.
Otherwise \[ \gr_{r'}(p^2 \Delta+\lambda_0) = \frac{1}{2}H^2+2EF, \]
\[ \gr_{r'} D_r^\lambda(G, \QQ_p) \simeq \mathbb{F}_{p}[\epsilon,\epsilon^{-1}][F,H,E] / (\frac{1}{2}H^2+2EF) \] is an integral domain by the proof of the third part of Lem \ref{Ug quot cent}.
\end{proof}

\begin{rem}
\begin{itemize}
\item Although $\gr_{1/p} D_r(G, \QQ_p)$ is noncommutative, it can probably be shown that $\gr_{1/p} D_r^\lambda(G, \QQ_p)$ is an integral domain as well.
\item The infinitesimal character $\lambda$ is induced if and only if the valuation $\lambda_0$ is at least $2$ since for any continuous character $\chi : T \simeq 1+p\ZZ_p \to 1+p\ZZ_p$, $d\chi(p^2 h^2) \in p^2\ZZ_p$.
\end{itemize}
\end{rem}

\begin{proof}[Proof of Theorem \ref{mic inf sp}]
For any non-zero $\delta \in \QQ_p[[G]]$, there exists $n_{\delta} \geq 1$ such that the image of $\delta$ in $D_r^\lambda(G, \QQ_p)$ is non-zero for $r = \sqrt[p^{n_\delta}]{1 / p}$ by Thm \ref{dis inf sp} and Prop \ref{D inv lim iso}.
The right $D_r^\lambda(G, \QQ_p)$-module $D_r^\lambda(G, \QQ_p) / \delta \cdot D_r^\lambda(G, \QQ_p)$ is torsion.
Since $D_r^\lambda(G, \QQ_p)$ is an integral domain by Thm \ref{Dr domain}, $D_r^\lambda(G, \QQ_p) / \delta \cdot D_r^\lambda(G, \QQ_p)$ has positive codimension by \cite[Prop 2.5]{AW13}.
By applying Prop \ref{crossed prod Dr} and \cite[Cor 5.4]{AB07}, $D_r^\lambda(G, \QQ_p) / \delta \cdot D_r^\lambda(G, \QQ_p)$ also has positive codimension over $U_r^\lambda(\frak{g})$.
By Lem \ref{Ug quot cent}, $U_r^\lambda(\frak{g})$ is an integral domain, there exists an element $\delta' \in D_r^\lambda(G, \QQ_p)$ such that $\delta \delta' \in U_r^\lambda(\frak{g})$ is non-zero.
By the description (\ref{Ur conver rate}) of $U_r(\frak{g})$, we leave the reader to prove
\[ U_r^\lambda(\frak{g}) = \left\{\sum_{\alpha} d_{\alpha} \prod_{i=1}^m e_i^{\alpha^i_a} f_i^{\alpha^i_b} h_i^{\alpha^i_c} \mid \alpha^i_c \in \{ 0, 1\}, ~ d_{\alpha} \in K, \lim_{|\alpha| \rightarrow \infty} | d_{\alpha} | \left\|\prod_{i=1}^m e_i^{\alpha^i_a} f_i^{\alpha^i_b} h_i^{\alpha^i_c}\right\|_{r}=0\right\} \]
similar to the first part of Lem \ref{Ug quot cent} giving topological basis of $U_r^\lambda(\frak{g})$, and similarly for $\widehat{U(\frak{g})} \otimes_{\Zg,\lambda} \QQ_p$.
Given the basis for both $U_r^\lambda(\frak{g})$ and $\widehat{U(\frak{g})} \otimes_{\Zg,\lambda} \QQ_p$, it is direct to see that the natural inclusion \[ U_r^\lambda(\frak{g}) \hookrightarrow \widehat{U(\frak{g})} \otimes_{\Zg,\lambda} \QQ_p \] is injective, hence $\delta \delta'$ is non-zero in $\widehat{U(\frak{g})} \otimes_{\Zg,\lambda} \QQ_p$.
The completed enveloping algebra is identified with $D_{1/p}(G, \QQ_p)$ (\cite[Rem 10.5, (c)]{AW13}, \cite[Lem 5.2]{AW14}).
The image of $\delta$ via the microlocalisation (\ref{micro}) is non-zero as well.
\end{proof}

\section{Local applications to finitely generated Iwasawa modules}\label{loc app}
We continue to use notation from the previous sections.

Let $p$ be an odd prime, $r \geq 1$ be a positive integer.
For each $1 \leq i \leq r$, let \[ G_i := (I_2 + p \mathrm{M}_2(\ZZ_p)) \cap \SL_2(\ZZ_p). \]
Let $G = \prod\limits_{i=1}^r G_i$, $\bk = (k_1,\cdots,k_r) \in \NN^r$, and $W_{\bk}$ be the algebraic representation $\boxtimes_{i=1}^{r} \Sym^{k_i}$ of $G$.
As $G$ is compact, by choosing an integral lattice, $W_{\bk}$ is equipped with a structure of finite dimensional Banach representation of $G$.
We leave an exercise for the reader to show $W_\bk$ is irreducible self-dual.
By main result of \cite{ST02I}, $W_{\bk}^\ast \simeq W_{\bk}$ admits an action of $K[[G]]$, and we do not distinguish $W_{\bk}^\ast$ and $W_{\bk}$.

We exhibit the $W_{\bk}$-quotient of $K[[G]]$ explicitly using the theory of Schneider-Teitelbaum as follows.
To ease notation, we do this for $r=1$, $\bk = k \in \NN$.
As a module over itself, the dual (constructed in \cite{ST02I}) of $K[[G]]$ is the Banach representation of continuous function $C(G, K)$ on $G$.
Our choice of $G$ can be viewed as an open subgroup of $\ZZ_p$-points of the group scheme $\SL_2$ over $\ZZ_p$. 
We define the space of algebraic vectors $C^{\mathrm{alg}}(G, K)$ of $C(G, K)$ to be the following $K$-linear vector space of polynomial functions on $G$ \[ \{ K[x,y,z,w]/(xw-yz-1) \}, \] as $\begin{pmatrix}
x & y \\
z & w
\end{pmatrix} \in G \subset \SL_2(\ZZ_p)$.
There is a natural two sided action of $G$ on $C^{\mathrm{alg}}(G, K)$ with a $G$-stable filtration $C^{\mathrm{alg}}_{\leq n}(G, K)$ by degrees of polynomials in variables $x,y,z,w$ such that \[ \dim_K (C^{\mathrm{alg}}_{\leq n}(G, K)/C^{\mathrm{alg}}_{\leq n-1}(G, K)) = \binom{n+3}{3} - \binom{n+1}{3} = (n+1)^2. \]
Here the dimension of the space of homogenous polynomials of degree $n$ (resp. $n-2$) in $4$-variables is $\binom{n+3}{3}$ (resp. $\binom{n+1}{3}$).

The space of degree $k$ homogenous polynomials in variables $x, z$ is isomorphic to $W_k$ as a left representation of $G$.
The vector $x^k$ is left invariant under the lower triangular nilpotent subgroup and right invariant under the lower upper triangular nilpotent subgroup with the highest weight in $W_k$, it generates an irreducible $G \times G$ subrepresentation $W_k \otimes W_k^\ast$ inside $C^{\mathrm{alg}}_{\leq k}(G, K)/C^{\mathrm{alg}}_{\leq k-1}(G, K)$.
By counting the dimensions we have \[ C^{\mathrm{alg}}_{\leq k}(G, K)/C^{\mathrm{alg}}_{\leq k-1}(G, K) \simeq W_k \otimes W_k^\ast. \]
By complete reducibility of finite dimensional representations, $W_k \otimes W_k^\ast$ is identified with $W_k$-algebraic vectors in $C(G, K)$ as a left or right representation of $G$, it is automatically closed in $C(G, K)$ by finite dimensionality.
Again by the main result of \cite{ST02I}, $W_k \otimes W_k^\ast$ is also identified with the maximal $W_k$-quotient of $K[[G]]$. 

Note that as an algebraic representation of $G$, $W_{\bk}$ receives a Lie algebra action (with respect to the explicit algebraic structure of $G$ above). 
$W_{\bk}$ remains irreducible when restricted to arbitrarily small open subgroups of $G$.
\cite[Thm 11.3]{AW13} says any finite dimensional simple Iwasawa module is a tensor product of a smooth $G$-representation and a Lie module (\S \ref{alg quo}).
Therefore our $W^\ast_\bk \simeq W_\bk$ is a Lie module,
we are entitled to apply results in \S \ref{alg quo}.

Let $\wAnn(\bk) := \Ann_{K[[G]]} (W_{\bk}) = \widetilde{\Ann(W_{\bk})}$ be the annihilator ideal of $K[[G]]$ for $W_{\bk}$, we have an isomorphism of algebras \begin{eqnarray}\label{KG Ann} K[[G]] / \wAnn(\bk) \xrightarrow{\sim} \End_K (W_{\bk}). \end{eqnarray}
Here the surjectivity follows from either directly counting dimensions of both sides given our explicit computation for $\SL_2$, or combining (\ref{KG iso}) and Thm \ref{KG quo}.

\begin{proof}[Proof of Theorem \ref{main loc thm}]
Any $\QQ_p[[G]]$-equivariant map from $\wtM$ to $W_{\bk}$ factors through $\wtM_{W_{\bk}}$, we have $\Hom_{\QQ_p[[G]]}(\wtM, W_{\bk}) = \Hom_{\QQ_p[[G]]}(\wtM_{W_{\bk}}, W_{\bk})$.
By Thm \ref{KG quo}, \[ \Hom_{\QQ_p[[G]]}(\wtM, W_{\bk}) \simeq \Hom_{\wUgK}(\whM, W_{\bk}). \]
There are multiple ways to express the multiplicity of $W_{\bk}$ in $\wtM$, we have
\[ \Hom_{\QQ_p[[G]]} (\wtM, W_{\bk}) \simeq \Hom_{\QQ_p[[G]]} (\wtM \otimes W_{\bk}^\ast, \mathds{1}) \simeq H_0(G, \wtM \otimes W_{\bk}^\ast)^\ast. \]
The dimensions of these $K$-vector spaces all agree with $H^0_{\whM}(\bk)$ in (\ref{Wk mul}).

We choose short exact sequences 
\begin{eqnarray}\label{tilde Q} 0 \to (\QQ_p[[G]])^d \to \wtM \to \widetilde{Q} \to 0, \end{eqnarray}
\begin{eqnarray}\label{tilde N} 0 \to \widetilde{N} \to (\QQ_p[[G]])^l \to \wtM \to 0, \end{eqnarray}
where $\widetilde{Q}$ is torsion.
Let $\whM = \widehat{U(\frak{g})} \otimes_{K[[G]]} \wtM$ be the microlocalisation of $\wtM$, and similarly for $\widetilde{Q}$, $\widetilde{N}$.
The microlocalisations preserve short exact sequences like (\ref{tilde Q}), (\ref{tilde N}) as (\ref{micro}) is flat.
We may filter $\widetilde{Q}$ by cyclic subquotients $\widetilde{Q}_1, \cdots, \widetilde{Q}_q$ with $\QQ_p[[G]] / \QQ_p[[G]] \cdot \delta_i \twoheadrightarrow \widetilde{Q}_i$ with non-zero $\delta_i \in \QQ_p[[G]]$.
By Thm \ref{mic inf sp}, $\delta_i$ are generic (\S \ref{gro enve}).
Thm \ref{upper bound} produces $p_{\widehat{Q}_i}$ such that $H^0_{\widehat{Q}_i}(\bk) \leq p_{\widehat{Q}_i}(\bk)$.
And we may choose $p_{\widehat{Q}} := \sum_{i=1}^q p_{\widehat{Q}_i}$ so $H^0_{\widehat{Q}}(\bk) \leq p_{\widehat{Q}}(\bk)$.
By applying the microlocalisations to the sequence (\ref{tilde Q}) followed by taking $W_\bk$ quotient
\[ (\widehat{U(\frak{g})} / \widehat{\Ann(\bk)})^d \to \whM_\bk \to \widehat{Q}_\bk \to 0, \] 
we have \[ H^0_{\whM}(\bk) \leq d \prod_{i=1}^r (k_i+1) + p_{\widehat{Q}}(\bk), \] therefore proving one side of $i=0$ case for any finitely generated $\QQ_p[[G]]$-module $\wtM$.
Similarly by applying the microlocalisations to the sequence (\ref{tilde N}) followed by taking the $W_\bk$ quotient, we get the other side of inequality for the case $i=0$ using \[ \rank_{\wUg} \widehat{N} + \rank_{\wUg} \whM = l. \]

Because $(\QQ_p[[G]])^l$ is acyclic, the long exact sequence associated to (\ref{tilde N}) in homology gives \[ 0 \to H_1(G,\wtM \otimes W_{\bk}) \to H_0(G, \widetilde{N} \otimes W_{\bk}) \to \] \begin{eqnarray}\label{LSE wtM} H_0(G, (\QQ_p[[G]] / \wAnn(\bk)) \otimes W_{\bk})^l \to H_0(G,\wtM \otimes W_{\bk}) \to 0, \end{eqnarray}
\begin{eqnarray}\label{ind coh} H_i(G, \wtM \otimes W_{\bk}) \simeq H_{i-1}(G, \widetilde{N} \otimes W_{\bk}), ~~ i \geq 2. \end{eqnarray}

The case $i=1$ is obtained by applying case $i=0$ to both $\wtM$ and $\widetilde{N}$ using (\ref{LSE wtM}). And $i \geq 2$ cases follow from an induction using (\ref{ind coh}).
\end{proof}

\section{Global automorphic applications}
In this final section, we refer to notation in the introduction section \S \ref{intro}.
We choose an odd prime $p$ which splits completely in $F$.
Recall that $r := [F:\QQ]$.
Let \[ G := \prod\limits_{v | p} G_v, ~ G_v := (I_2 + p \mathrm{M}_2(\ZZ_p)) \cap \SL_2(\ZZ_p), \] as the ring of integers in $F_v$ is identified with $\ZZ_p$ for all $v | p$.

For $\bk \in \NN^r$, it defines a $\CC$-representation $V_{\bk}$ of $\SL_2(F_{\infty})$, as well as a $\QQ_p$-representation $W_{\bk}$ of $\SL_2(F_p)$. 
These representations give rise to a $\CC$-local system $V_{\bk}$ and its corresponding $\QQ_p$-local system $W_{\bk}$.
It is explained in \cite[\S 5]{Mar12} that \begin{eqnarray}\label{C Qp comp} \dim_{\CC} H_i(Y(K_f), V_{\bk}) = \dim_{\QQ_p} H_i(Y(K_f), W_{\bk}). \end{eqnarray}

We cite some properties of the completed homology (\ref{comp homo})
\[ \wH_{\bullet}(K^p):=\varprojlim_{s} \varprojlim_{K_p \subset G} H_{\bullet} (Y(K_pK^p), \ZZ / p^{s} \ZZ) \otimes_{\ZZ_p} \QQ_p, \]
from \cite{Eme06}, \cite{CE09}: 
\begin{itemize}
\item If $F$ is not totally real, $\wH_{i}(K^p)$ is a finitely generated torsion $\QQ_p[[G]]$-module.
\item There is a spectral sequence
\begin{eqnarray}\label{comp SS}
E_{2}^{i, j} = H_{i} (G, \wH_{j}(K^p) \otimes W_{\bk}) \Longrightarrow H_{i+j} (Y(K_f), W_{\bk}),
\end{eqnarray}
where $K_f = G K^p$.
\end{itemize}

By (\ref{C Qp comp}), the spectral sequence (\ref{comp SS}) implies an upper bound \[ \dim_{\CC} H_{q} (Y(K_f), V_{\bk}) \leq \sum_{i+j=q} \dim_{\QQ_p} H_{i}(G, \wH_{j}(K^p) \otimes W_{\bk}). \]
Thm \ref{main loc thm} produces a multiplicity-free polynomial $p_{K_f}$ of degree at most $r-1$ for the right hand side, giving \[ \dim_{\CC} H_{q} (Y(K_f), V_{\bk}) \leq p_{K_f}(\bk), ~ \bk \in \NN^r. \]
Finally, we apply the Poincar\'e duality and Eichler-Shimura isomorphism (\ref{ES}) to finish our proof of Thm \ref{main uncond}.

\end{document}